\documentclass[11pt,a4paper,reqno]{amsart}
\usepackage{amsmath,amsthm,amssymb,a4wide}
\usepackage[utf8]{inputenc}
\usepackage{graphicx}
\usepackage{tikz}
\usepackage{pgfplots}
\usepackage{hyperref}
\usepackage{pgfplotstable}
\newcommand{\eps}{\varepsilon}
\newcommand{\Se}{\mathcal{S}^\eps}
\newcommand{\ueh}{u^{\eps,h}}

\newcommand{\eh}{e^{\eps,h}}
\newcommand{\De}{D^\eps}
\newcommand{\Deh}{D^\eps_h}
\newcommand{\Vehg}{V^\eps_{h,g}}
\newcommand{\Vehn}{V^\eps_{h,0}}

\newcommand{\Vh}{V_{h}}
\newcommand{\Wehg}{W^\eps_{h,g}}
\newcommand{\Wehn}{W^\eps_{h,0}}

\newcommand{\Wh}{W_{h}}
\newcommand{\Leh}{\mathcal{L}^{\eps}_h}
\newcommand{\T}{\mathcal{T}}
\newcommand{\RR}{\mathbb{R}}
\newcommand{\NN}{\mathbb{N}}
\def\d{\,\mathrm{d}}

\newtheorem{theorem}{Theorem}[section]
\newtheorem*{theorem*}{Theorem}
\newtheorem{lemma}[theorem]{Lemma}

\theoremstyle{remark}
\newtheorem{remark}[theorem]{Remark}

\begin{document}

\title[DDM for Dirichlet boundary value problems]{Error analysis of a diffuse interface method for elliptic problems with Dirichlet boundary conditions}
\author[M. Schlottbom]{Matthias Schlottbom$^{\dag}$}
\thanks{$^\dag$ Institute for Computational and Applied Mathematics,
University of Münster, Einsteinstr. 62, 48149 M\"unster, Germany.\\
Email: {\tt schlottbom@uni-muenster.de}
}
\date{\today}

\begin{abstract}
 We use a diffuse interface method for solving Poisson's equation with a Dirichlet condition on an embedded curved interface.  The resulting diffuse interface problem is identified as a standard Dirichlet problem on approximating regular domains.  We estimate the errors introduced by these domain perturbations, and prove convergence and convergence rates in the $H^1$-norm, the $L^2$-norm and the $L^\infty$-norm in terms of the width of the diffuse layer. For an efficient numerical solution we consider the finite element method for which another domain perturbation is introduced. These perturbed domains are polygonal and non-convex in general. We prove convergence and convergences rates in the $H^1$-norm and the $L^2$-norm in terms of the layer width and the mesh size. In particular, for the $L^2$-norm estimates we present a problem adapted duality technique, which crucially makes use of the error estimates derived for the regularly perturbed domains. Our results are illustrated by numerical experiments, which also show that the derived estimates are sharp.
\end{abstract}

\maketitle



\section{Introduction and main results}
This paper considers the approximate solution of the following model problem by a diffuse interface method: Find $u\in H^1_0(\Omega)$ such that
 \begin{align}\label{eq:model}
  -\Delta u = f \quad \text{in } \Omega\setminus \Gamma,\qquad 
%   u_{\mid\partial\Omega} = 0 \quad \text{on } \partial \Omega,\qquad
  u_{\mid\Gamma} = g_{\mid\Gamma} \quad \text{on } \Gamma.
 \end{align}
Here, the function $f\in L^2(\Omega)$ models volume sources in a convex polygonal domain $\Omega\subset\RR^n$, $n=2,3$, and the function $g\in H^2(\Omega)$ defines the values of $u$ on an interface $\Gamma\in C^{1,1}$, where $\Gamma \subset\Omega$ is a closed manifold of co-dimension one, i.e. for $n=2$ a curve, or a surface if $n=3$.
We assume that the interface $\Gamma$ separates $\Omega$ into two domains $\Omega=D_1\cup \Gamma\cup D_2$, where $\Gamma=\partial D_1$, see Figure~\ref{fig:domain}.

The analysis of \eqref{eq:model} is well-established.
For instance, if $g$ does not depend on $u$, \eqref{eq:model} can be separated into two independent Dirichlet problems on $D_1$ and $D_2$ respectively, and the theory for the Poisson equation with Dirichlet boundary conditions applies, cf. \cite{GT2001,Grisvard85}.
Alternatively, one may formulate \eqref{eq:model} as a Dirichlet problem on $\Omega$ constrained by $u=g$ on $\Gamma$ and $u=0$ on $\partial\Omega$, which leads to a saddle-point formulation, see e.g. \cite{BrezziFortin,GiraultGlowinski1995}.
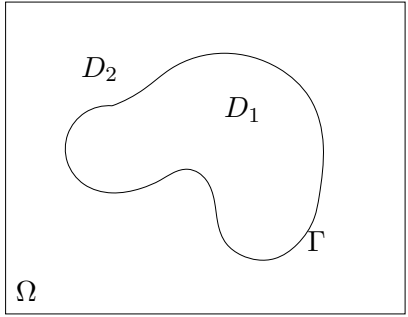
\begin{figure}
\setlength{\unitlength}{150pt}
 \begin{picture}(1,1)%
    \put(0,0){\includegraphics[width=\unitlength]{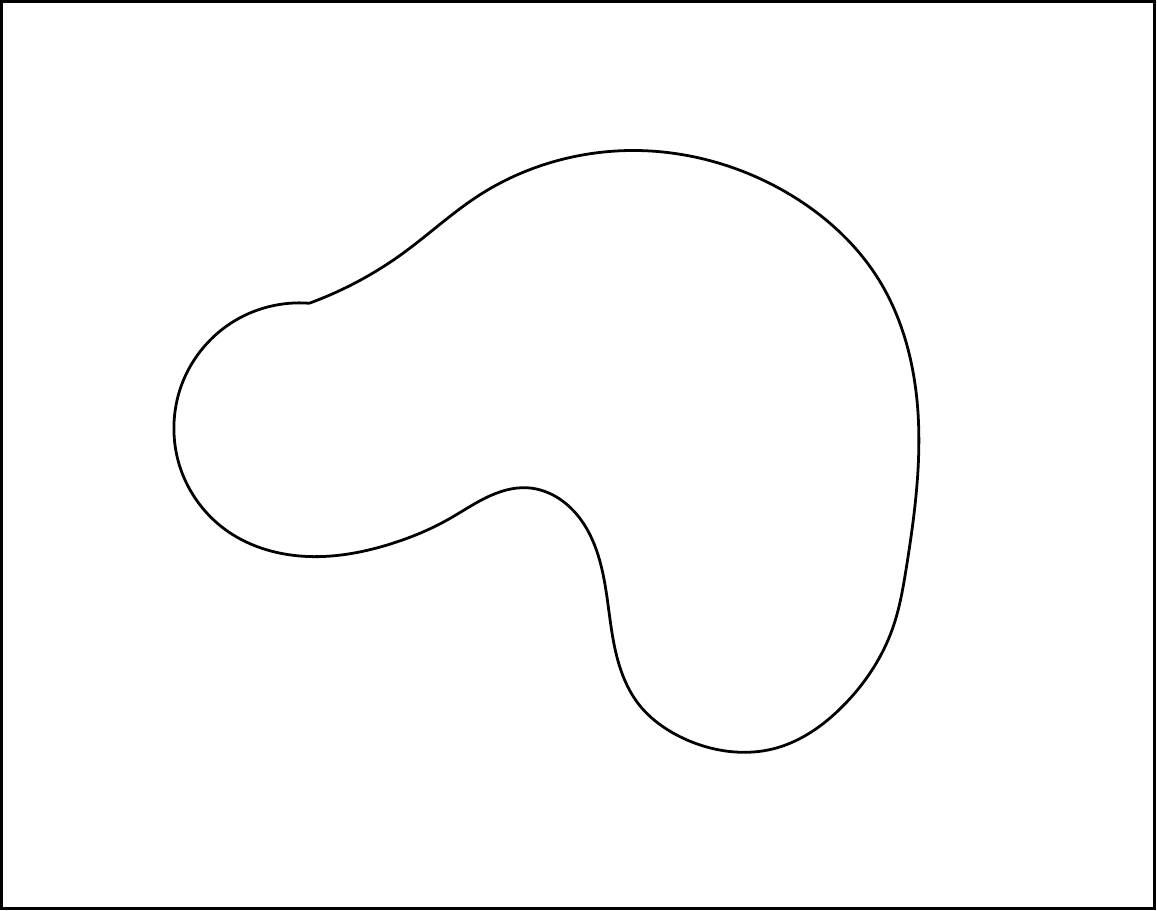}}%
    \put(0.02718471,0.03404955){$\Omega$}%
    \put(0.19022446,0.59696569){$D_2$}%
    \put(0.54719572,0.49570944){$D_1$}%
    \put(0.75485683,0.16080546){$\Gamma$}%
  \end{picture}%
  \caption{\label{fig:domain} Sketch of the geometry. $\Omega=D_1\cup\Gamma\cup D_2$.}
\end{figure}

The numerical approximation of \eqref{eq:model} has been investigated intensively.
A first set of numerical algorithms relies on a triangulation of $\Omega$ which is sufficiently aligned with the interface, i.e. $\Gamma$ is approximated by a polygon, the segments of which are edges (or faces) of the elements; see e.g. \cite{BarrettElliot1987,BrambleDupontThomee1972,BrambleKing1994,BrambleKing1997,Ciarlet2002,Thomee1973}. %; see also \cite{ElliottRanner2013}.
% Let us also mention \cite{ElliottRanner2013} where the Dirichlet condition on $\Gamma$ is replaced by a Robin type condition and an additional surface PDE.
The construction of such a triangulation might be expensive or difficult in practice. 
Furthermore, if $\Gamma=\Gamma(t)$ depends on time, problems similar to \eqref{eq:model} need to be solved in each time step, and the triangulation has to be updated accordingly. In addition, one has to interpolate the data on the varying meshes in this case.
Therefore, a lot of research has been conducted to construct accurate methods employing meshes which are not aligned with $\Gamma$ but possess a ``simple'' structure and are fixed throughout the simulation; see for instance
the immersed boundary method \cite{Peskin1977}, 
the immersed interface method \cite{LeVequeLin1994,LiLinWu2003}, 
immersed finite elements \cite{Li2003,LinLinZhang2015,WuLiLai2011},
the fictitious domain method \cite{Babuska72,GiraultGlowinski1995,GlowinskiPanPeriaux1994,MaitreTomas1999}, 
the unfitted finite element method \cite{BarrettElliot1987,ChenZou1998,HansboHansbo2002,LiMelenkWohlmuthZou2010}, 
the finite cell method \cite{ParvizianDuesterRank2007},
unfitted discontinuous Galerkin methods \cite{BastianEngwer2009},
or composite finite elements \cite{HackbuschSauter1997,LiehrPreusserRumpfSauterSchwen2009}, and the references provided there. 
% Furthermore, if e.g. \eqref{eq:model} is replaced by a time-dependent counterpart and $\Gamma$ varies in time, it may be expensive to interpolate solutions to different (time-dependent) meshes.

In this work we will focus on a diffuse interface method for solving \eqref{eq:model}, see for instance \cite{AbelsLamStinner2015,BES2014,ElliottStinner2009,ElliotStinnerStylesWelford2011,LervagLowengrub2014,LiLowengrubRaetzVoigt2009,ReuterHillHarrison2012}.
In this method the sharp interface condition $u=g$ on $\Gamma$ is replaced by suitable conditions on $u-g$ on a diffuse layer centered around $\Gamma$.
Similar techniques have also been applied for solving coupled bulk-surface differential equations \cite{AbelsLamStinner2015} or surface differential equations \cite{Bertalmio2001,Chernyshenko2013}.

The constraint $u=g$ on $\Gamma$ is equivalent to the condition $\|u-g\|_{L^2(\Gamma)}=0$.
In order to relax this condition on the sharp interface, we define the signed distance function
\begin{align*}
 d_\Gamma(x) = \begin{cases}
                -{\rm dist}(x,\Gamma),& x\in D_1,\\
                +{\rm dist}(x,\Gamma),& x\in D_2,
               \end{cases}
\end{align*}
and we let $S:\RR\to\RR$ be such that $S(t)=t$ for $|t|<1$ and $S(t)={\rm sign}(t)$ for $|t|\geq 1$.
The width of the diffuse layer is characterized by a positive parameter $\eps$. This induces a regularized indicator function of $D_1$
\begin{align*}
 \chi_{D_1}(x) \approx \omega^\eps(x) = \frac{1}{2}\big(1+S(-\frac{d_\Gamma(x)}{\eps})\big),\quad x\in \Omega.
\end{align*}
Formally $\d\Gamma = |\nabla \chi_D|\d x\approx |\nabla\omega^\eps|\d x$. 
Using the $\eps$-tubular neighborhood $\Se={\rm supp}(|\nabla\omega^\eps|)$ of $\Gamma$, this leads to the following approximation for integrals along the interface
\begin{align*}
 \|u-g\|_{L^2(\Gamma)}^2 \approx \frac{1}{2\eps} \int_{\Se}|u-g|^2 \d x.
\end{align*}
The reader might find a more detailed derivation of this approximation and other choices of $S$ in \cite{BES2014}.
We further notice that, to make this approximation well-defined, we need the function $g$ to be defined on the diffuse layer. If $g$ is defined on $\Gamma$ only, one has to use a suitable extension; for instance a local extension is given by $\tilde g(x+d_\Gamma(x)\nabla d_\Gamma(x))=g(x)$ for $x\in\Gamma$, i.e. $\tilde g$ is constant off the interface.
Thus, replacing the sharp interface constraint by the diffuse interface constraint $\int_{\Se}|u-g|^2 \d x=0$, which amounts to $u=g$ on $\Se$, we are concerned with the following Dirichlet problem: Find $u^\eps\in H^1_0(\Omega)$ such that
\begin{align}\label{eq:Dirichlet_diffuse}
  -\Delta u^\eps = f \quad \text{ in } \Omega\setminus\Se,\quad u^\eps=g\quad\text{in } \Se.
%   ,\quad u^\eps_{\mid\partial\Omega}=0\quad\text{on }\partial\Omega.
\end{align}
Note that the particular choice of $\omega^\eps$ is not important as long as $\Se={\rm supp}(|\nabla\omega^\eps|)$ is a $\eps$-tubular neighborhood of $\Gamma$, i.e. methods using double obstacle potentials to regularize the indicator function of $\chi_D$ will essentially lead to the same method.

The main purpose of this paper is to estimate the errors introduced by this diffuse interface method; (i) on the continuous level and (ii) in a finite dimensional setting when using the finite element method, see below.
The first result in this direction is the following approximation result on the continuous level.
\begin{theorem}\label{thm:main1}
 Let $f\in L^2(\Omega)$ and $g\in H^2(\Omega)$, and let $u$ be the solution to \eqref{eq:model}, and let $u^\eps$ be the solution to \eqref{eq:Dirichlet_diffuse}.
 Then there exists a constant $C>0$ independent of $\eps$ such that 
 \begin{align*}
  \frac{1}{\eps}\|u-u^\eps\|_{L^2(\Omega)} + \frac{1}{\sqrt{\eps}}\|\nabla u-\nabla u^\eps\|_{L^2(\Omega)} &\leq C\big( \|f\|_{L^2(\Omega)}+\|g\|_{H^2(\Omega)}\big).
%   \|u-u_\alpha^\eps\|_{L^2(\Omega)}&\leq C(\eps+\alpha)\big( \|f\|_{L^2(\Omega)}+\|F\|_{W^{\frac{3}{2},2}(\Gamma)}\big).
 \end{align*}
\end{theorem}

The numerical approximation of \eqref{eq:Dirichlet_diffuse} is still not straight-forward as the (sufficiently exact) integration over $\Se$ is basically not easier than the integration along $\Gamma$.
In order to obtain an efficient numerical scheme let $\T_h$ be a shape regular triangulation of $\Omega$, and let 
\begin{align}
  \Se_h=\bigcup_{T\in\T_h: \Se\cap T\neq \emptyset} T \label{eq:Seh}
\end{align}
denote the union of all elements having non-empty intersection with $\Se$; see Figure~\ref{fig:sketch} below for one instance of $\Se$ and $\Se_h$.
Here, $h=\max\{{\rm diam}(T): T\in\T_h\}$ denotes the mesh-size parameter.
We are then concerned with the following approximate Dirichlet problem: Find $u^{\eps,h}\in H^1_0(\Omega)$ such that
\begin{align}\label{eq:dirichlet_diffuse_h}
  -\Delta\ueh = f \quad \text{ in } \Omega\setminus\Se_h,\qquad \ueh=g\quad\text{in } \Se_h. % ,\quad \ueh_{\mid\partial\Omega}=0\quad\text{on }\partial\Omega.
\end{align}
Note that \eqref{eq:dirichlet_diffuse_h} is still formulated on the continuous level. The difference to \eqref{eq:Dirichlet_diffuse} is that we have replaced the domain $\Omega\setminus\Se$, which has the same smoothness as $\Omega\setminus\Gamma$, by $\Omega\setminus\Se_h$, which is a polygonal domain.
Problems similar to \eqref{eq:dirichlet_diffuse_h} for the approximation of \eqref{eq:model} have been considered previously e.g. in \cite{BrambleDupontThomee1972} for $n=2$; see also \cite{Blair1973} for very general approximating domains. We will employ similar techniques as in \cite{Blair1973} to derive Theorem~\ref{thm:main2} from Theorem~\ref{thm:main1}.
\begin{theorem}\label{thm:main2}
 Let $f\in L^2(\Omega)$ and $g\in H^2(\Omega)$, and let $u$ be the solution to \eqref{eq:model}, and let $\ueh$ be the weak solution to \eqref{eq:dirichlet_diffuse_h}.
 Then for some $C>0$ independent of $\eps$, $h$, and $\delta$ there holds
 \begin{align*}
  \frac{1}{\eps+\delta}\|u-\ueh\|_{L^2(\Omega)} + \frac{1}{\sqrt{\eps+\delta}}\|\nabla u-\nabla \ueh\|_{L^2(\Omega)} &\leq C\big( \|f\|_{L^2(\Omega)}+\|g\|_{H^2(\Omega)}\big),
 \end{align*}
 where $\delta = \max\{{\rm diam}(T): T\cap\partial\Se\neq \emptyset,\ T\in \T_h\}$.
\end{theorem}
Concerning the numerical approximation of \eqref{eq:dirichlet_diffuse_h}, we denote by $\Vh=P_1(\T_h)\cap H^1_0(\Omega)$ a standard finite element space of continuous piecewise affine functions.
For the incorporation of the interface values we denote by $I_hg \in P_1(\T_h)\cap H^1(\Omega)$ the nodal interpolant of $g$, and let
\begin{align}
 \Vehg=\{v_h \in \Vh: v_{h\mid T}=(I_h g)_{\mid T} \text{ for } T\cap \Se\neq \emptyset\}.
\end{align}
The Galerkin approximation $u^\eps_h\in\Vehg$ of $\ueh$ is defined by the variational problem
\begin{align}
 \int_{\Omega}\nabla u^\eps_h \cdot\nabla v_h \d x = \int_{\Omega} f v_h \d x\quad\text{for all } v_h\in \Vehn.\label{eq:Galerkin}
\end{align}
By construction $\Se_h$ is aligned with $\T_h$. Therefore, integration of the bilinear form in \eqref{eq:Galerkin} can be performed by standard techniques without additional variational crimes.
Note also, that the extension of $g$ is only needed on $\Se$.
Using Ce\'a's lemma, Theorem~\ref{thm:main2} and problem adapted duality arguments, we will also prove the following theorem.
\begin{theorem}\label{thm:main3}
Let $f\in L^2(\Omega)$ and $g\in H^2(\Omega)$, and let $u$ be the solution to \eqref{eq:model}, and let $u_h^\eps \in\Vehg$ be defined via \eqref{eq:Galerkin}
 Then for some $C>0$ independent of $\eps$, $\delta$, $\kappa$, and $h$ there holds
 \begin{align*}
  \frac{1}{\lambda^2}\|u-u^\eps_h\|_{L^2(\Omega)} + \frac{1}{\lambda}\|\nabla u-\nabla u_h^\eps\|_{L^2(\Omega)} &\leq C\big( \|f\|_{L^2(\Omega)}+\|g\|_{H^2(\Omega)}\big),
 \end{align*}
 where $\lambda=\sqrt{\eps+\delta}+\kappa^{\frac{2}{3}}+h$,
 $\kappa = \max\{{\rm diam}(T): T\cap\partial\mathcal{S}^{\eps+\delta}\neq \emptyset, \ T\in \T_h\}$ and $h = \max\{{\rm diam}(T): T\in \T_h\}$.
\end{theorem}
In view of the above theorems, the diffuse interface method investigated here is robust in the choice of $\eps$ and $h$. In particular, $\eps$ might be chosen arbitrarily small such that the overall approximation error is governed by approximation errors from the finite element approximation. 
In this case the described method is similar to the method described in \cite{LiMelenkWohlmuthZou2010}.
However, we do not impose the restriction that our triangulation is $\delta$-resolved or $\eps$-resolved in the sense of \cite[Definition~3.1]{LiMelenkWohlmuthZou2010}, and, also, $\partial\Se_h\subset \mathcal{S}^\delta$ basically induces two polygonal interfaces approximating $\Gamma$, whereas in \cite{LiMelenkWohlmuthZou2010}, one polygonal interface is defined. The $H^1$-estimate in Theorem~\ref{thm:main3} has a similar form as the estimate proven in \cite[Theorem~4.1]{LiMelenkWohlmuthZou2010}, but here the value of $\delta$ is already determined by the mesh size. The value of $\delta$ in \cite{LiMelenkWohlmuthZou2010} stems from the assumption that the triangulation is $\delta$-resolved. %Using, locally refined triangulations such that $\delta=O(h^2)$, we can recover the rate in \cite[Theorem~4.1]{LiMelenkWohlmuthZou2010}.
If the finite element grid is aligned with $\Gamma$, i.e. if the vertices defining $\partial\Se_h$ are points on $\partial\Se$, stronger approximation error estimates can be derived \cite{BarrettElliot1987,BrambleKing1994,BrambleKing1997,Thomee1973}.

Let us compare our results with those in the literature about diffuse domain methods for solving the Dirichlet problem \eqref{eq:model} in $D_1$. In \cite{BES2014} the following variational problem has been considered: Find $\tilde u^\eps_\alpha\in H^\eps$ such that
\begin{align}\label{eq:BES2014}
 \int_{\Omega} \nabla \tilde u^\eps_\alpha \cdot \nabla v \omega^\eps \d x +  \frac{1}{\alpha}\int_\Omega  \tilde u^\eps_\alpha v|\nabla\omega^\eps| \d x= \int_\Omega f v \omega^\eps \d x+  \frac{1}{\alpha}\int_\Omega g v|\nabla\omega^\eps| \d x
\end{align}
for all $v\in H^\eps$, where $H^\eps$ is a weighted Sobolev space, which consists of functions in the weighted space $L^2(\Omega;\omega^\eps)$ having also weak derivatives in this space; details are provided in \cite{BES2014}. Note that functions in $H^\eps$ can develop singularities where $\omega^\eps=0$. 
The analysis in \cite{BES2014} requires a careful balance between $\alpha$ and $\eps$. Provided $u\in W^{3,\infty}(D_1)$, the optimal choice of $\alpha=\eps^{3/4}$ gives a rate $\|\tilde u^\eps_\alpha-u\|_{H^1(D_1)} = O(\eps^{3/4})$ in \cite{BES2014}.
Let us point out, that, under additional assumptions, a rate $O(\eps)$ for the $L^2$ error could be obtained for $n\leq 2$. If $u\notin W^{3,\infty}(D_1)$ but only in $H^2(D_1)$, which is essentially what we need here, the convergence rates of \cite{BES2014} get worse and depend on the dimension $n$ due to embedding theorems.
However, the computations in \cite{ReuterHillHarrison2012} show a convergence for the $L^2$-error of order $O(\eps)$ for a similar diffuse domain method, where the diffuse interface condition is incorporated by a penalty method similar to \eqref{eq:BES2014}; see Remark~\ref{rem:saddle-point} below.
In \cite{FranzGaertnerRoosVoigt2012} an $L^\infty$-error estimate of order $\eps^{s}$ for $s<1$ arbitrary has be proven for a diffuse domain method with a double well regularization for $\chi_{D_1}$ and $n=1$. As proven in Theorem~\ref{thm:errorlinfty_h} below, the diffuse interface method considered here yields $\|u-\ueh\|_{L^\infty(\Omega\setminus \Se_h)}= O(\eps+\delta)$. 
We consider only Dirichlet boundary conditions here, for Neumann or Robin boundary conditions let us refer to \cite{BES2014}, where under appropriate regularity conditions on the data even superlinear convergence has been shown for the diffuse domain method on the continuous level.
 
The rest of this paper is structured as follows. In Section~\ref{sec:dirichlet_revisited} we introduce further notation and recall some solvability and regularity results for \eqref{eq:model} and \eqref{eq:Dirichlet_diffuse}.
% as well as the Ce\'a lemma and the Aubin-Nitsche lemma.
In Section~\ref{sec:continuous_error} we derive $H^1$- and $L^2$-error estimates, which prove Theorem~\ref{thm:main1}.
Theorem~\ref{thm:main2} is proven in Section~\ref{sec:interface}. For completeness we also give an $L^\infty$-error estimate.
Galerkin approximations are investigated in Section~\ref{sec:discrete_error}, where also Theorem~\ref{thm:main3} is proved.
Our findings are supported by numerical examples in Section~\ref{sec:numerics}.
Section~\ref{sec:conclusion} gives some conclusions. The paper closes with an appendix recalling some estimates for the errors introduced by the diffuse interface method.

\section{The Dirichlet problem revisited}\label{sec:dirichlet_revisited}
Let $\tilde\Omega\subset\RR^n$ be some bounded domain with Lipschitz boundary.
By $L^2(\tilde \Omega)$ we denote the Lebesgue space of square integrable functions. Accordingly, for an integer $k\geq 1$, $H^k(\tilde\Omega)=W^{k,2}(\tilde\Omega)$ is the set of functions in $L^2(\tilde\Omega)$ having also weak derivatives up to order $k$ in $L^2(\tilde\Omega)$. These spaces are endowed with the standard norms, i.e.
\begin{align*}
 \|u\|_{L^2(\tilde\Omega)}^2= \int_{\tilde \Omega} |u|^2 \d x,\qquad
 \|u\|_{H^k(\tilde\Omega)}^2= \sum_{|\alpha|=0}^k \|D^\alpha u\|_{L^2(\tilde\Omega)}^2,
\end{align*}
where $\alpha\in\NN_0^n$ is a multi-index. We will write $\nabla u$ for the gradient of $u$.
$H^1_0(\tilde\Omega)$ is the closure in $H^1(\tilde\Omega)$ of infinitely often differentiable functions with compact support in $\tilde \Omega$.
We recall the Poincar\'e inequality \cite[6.26]{Adams1975}
\begin{align*}
 \| v\|_{L^2(\tilde \Omega)} \leq \frac{{\rm diam}(\tilde\Omega)}{\sqrt{2}} \|\nabla v\|_{L^2(\tilde \Omega)}\quad\text{for all } v\in H^1_0(\tilde\Omega).
\end{align*}
Hence, $\|\nabla v\|_{L^2(\tilde\Omega)}$ defines an equivalent norm on $H^1_0(\tilde \Omega)$.
By $u_{\mid \tilde\Omega}$ we denote the restriction of a function $u$ to $\tilde\Omega$; if the context is clear, we will write $u$ instead of $u_{\mid \tilde\Omega}$.

In the whole manuscript we let $\eps_0>0$ be such that (i) the projection of the diffuse layer onto $\Gamma$ is well-defined, see \eqref{eq:projection}, and (ii) \eqref{eq:det_half} holds.
Furthermore, we let $0<\eps\leq \eps_0$. %, where $\eps_0$ depends on the curvature of $\Gamma$; see \cite{BES2014} and the appendix.
We denote by $C$ a generic constant which may change from line to line, and, in particular, $C$ is independent of $\eps$, $\delta$ and $h$ and the data.
Under the given assumptions, we have the following existence and regularity result, which follows from a combination of results from standard theory on elliptic equations \cite{GT2001,Grisvard85}.
\begin{lemma}\label{lem:regularity_sharp}
 Let $\Gamma\in C^{1,1}$, $f\in L^2(\Omega)$, $g\in H^2(\Omega)$. Then there exists a unique solution $u\in H^1_0(\Omega)$ to \eqref{eq:model} such that $u_{\mid D_i}\in H^2(D_i)$, $i=1,2$, and
 $$
  \|u\|_{H^1(\Omega)}+\|u_{\mid D_1}\|_{H^2(D_1)}+\|u_{\mid D_2}\|_{H^2(D_2)} \leq C (\|f\|_{L^2(\Omega)} + \|g\|_{H^2(\Omega)}).
 $$
\end{lemma}
% \begin{proof}
%  Using Poincar\'e's inequality and the Lax-Milgram lemma, we obtain existence and uniqueness of $u_1 \in H^1(D_1)$ such that $u_1-g\in H^1_0(D_1)$ and $u_1$ is the weak solution to $-\Delta u_1=f$ in $D_1$. Similarly, there exists a unique solution $u_2 \in H^1(D_2)$ of $-\Delta u_2=f$ in $D_2$ such that $u_2=0$ on $\partial \Omega$ and $u_2=g$ on $\Gamma$ in the sense of traces. Since $u_1=u_2$ on $\Gamma$, we can define $u\in H^1_0(\Omega)$ by setting $u_{\mid D_i}=u_i$ for $i=1,2$. Furthermore, there exists a constant $C>0$ such that $\|u\|_{H^1(\Omega)} \leq C(\|f\|_{L^2(\Omega)} + \|g\|_{H^2(\Omega)})$. Since, $D_1$ is a $C^{1,1}$ domain and $\Omega$ is a convex domain, we further have that $u_{\mid D_i} \in H^2(D_i)$, $i=1,2$ \cite{GT2001,Grisvard85}, and $\|u_{\mid D_i}\|_{H^2(D_i)} \leq C (\|f\|_{L^2(\Omega)} + \|g\|_{H^2(\Omega)})$.
%  % Theorem 3.2.1.2 Grisvard85: Convex domains are uniformly approximated by C^2 domains.
% \end{proof}
%
For the solution of the diffuse interface problem we have
\begin{lemma}\label{lem:regularity_diffuse}
 Let $\Gamma\in C^{1,1}$, $f\in L^2(\Omega)$, $g\in H^2(\Omega)$. 
 Then there exists a unique solution $u^\eps\in H^1_0(\Omega)$ to \eqref{eq:Dirichlet_diffuse} such that $u^\eps_{\mid D_i\setminus\Se}\in H^2(D_i\setminus\Se)$, $i=1,2$, and
 $$
  \|u^\eps\|_{H^1(\Omega)}+\|u_{\mid D_1\setminus\Se}^\eps\|_{H^2(D_1\setminus\Se)}+\|u_{\mid D_2\setminus\Se}^\eps\|_{H^2(D_2\setminus\Se)} \leq C (\|f\|_{L^2(\Omega)} + \|g\|_{H^2(\Omega)}).
 $$
\end{lemma}
% \begin{proo

For later reference, we note that the solutions $u$ to \eqref{eq:model} and $u^\eps$ to \eqref{eq:Dirichlet_diffuse}, satisfy
\begin{align}\label{eq:trace_of_diffuse_u}
 u^\eps-u=0 \quad\text{on } \Gamma.
\end{align}

% We furthermore need the following tools, see e.g. \cite[Theorem~2.8.1]{BrennerScott} and \cite[Theorem~3.2.4]{Ciarlet2002}.
% \begin{lemma}[Ce\'a Lemma]\label{lem:cea}
%  Let $V$ be a Hilbert space and let $a:V\times V\to \RR$ be a continuous and coercive bilinear form. Let $\ell\in V'$ be a bounded linear functional, and let $u\in V$ satisfy
%  $$
%   a(u,v) = \ell(v)\quad\text{for all } v\in V.
%  $$
%  Let $W\subset V$ be a closed subspace, and let $u_W\in W$ be the solution to
%  $$
%   a(u_W,v) = \ell(v)\quad\text{for all } v\in W.
%  $$
%  Then, there exists a constant depending on $a$ only such that
%  \begin{align*}
%   \|u-u_W\|_V \leq C \inf_{v\in W} \|u-v\|_V.
%  \end{align*}
% \end{lemma}
% \begin{lemma}[Aubin-Nitsche Lemma]\label{lem:AubinNitsche}
%  Let the assumptions of Lemma~\ref{lem:cea} hold true. Furthermore, let $H$ be a Hilbert space such that $V\subset H$, dense with continuous embedding, then there exists a constant $C$ depending on $a$ only such that
%  \begin{align*}
%   \|u-u_W\|_H^2 \leq C \|u-u_W\|_V \inf_{v\in W} \| z- v\|_{V}
%  \end{align*}
%  where $z\in V$ is the solution to
%  $$
%   a(v,z) = \langle u-u_W, v\rangle_H\quad\text{for all } v\in V.
%  $$
% \end{lemma}

\section{Interface problems involving $\Se$}\label{sec:continuous_error}
% In this section we prove Theorem~\ref{thm:main1}, which is a consequence of Theorem~\ref{thm:errorH1} and Theorem~\ref{thm:L2estimate} below.
We start with error estimates for the solutions to the model problem \eqref{eq:model} and its diffuse interface reformulation \eqref{eq:Dirichlet_diffuse}.
First, we derive estimates in the $H^1$-seminorm, and then we derive a corresponding $L^2$-error estimate via the Aubin-Nitsche lemma. %~\ref{lem:AubinNitsche}.
Theorem~\ref{thm:main1} is then a consequence of Theorem~\ref{thm:errorH1} and Theorem~\ref{thm:L2estimate} below.
% A combination of Theorem~\ref{thm:errorH1} and Theorem~\ref{thm:L2estimate} then proves Theorem~\ref{thm:main1}.
Since \eqref{eq:model} is a linear equation, we may assume without loss of generality that $g=0$ in this section.
For the sake of simplicity, we let $D=D_1$ and $D^\eps=D\setminus\Se$ in the rest of this section, and perform the analysis for this domain. In particular $\Gamma=\partial D$. The remaining estimates on $D_2$ can be derived analogously.
Furthermore, in slight abuse of notation, we denote by $u \in H^1_0(D)\cap H^2(D)$ the restriction of the solution to \eqref{eq:model} to $D$. %, which exists by Lemma~\ref{lem:regularity_sharp}.
By $u^\eps\in H^1_0(\De)\cap H^2(\De)$ we denote the corresponding restriction of the solution to \eqref{eq:Dirichlet_diffuse}. %, which exists by Lemma~\ref{lem:regularity_diffuse}.
In particular, there holds
\begin{align}
 -\Delta u &= f \quad \text{in } D,\quad u=0\quad\text{on } \partial D,\label{eq:model1_wlog}\\
 -\Delta u^\eps &= f \quad \text{in } \De,\quad u=0\quad\text{on } \partial \De.\label{eq:model2_wlog}
\end{align}
% \subsection{$H^1$-estimates}
%
Hence, we have for all $v\in H^1_0(D)$ that 
\begin{align}\label{eq:u_weak}
 \int_{D} \nabla u \cdot\nabla v\d x= \int_{D} f v\d x.
\end{align}
% and similar
Extending $u^\eps=0$ on $\Gamma_\eps=D\setminus\De$ and
integration by parts shows that for any $v\in H^1_0(D)$
\begin{align}\label{eq:ue_weak}
  \int_D \nabla u^\eps\cdot \nabla v \d x=\int_{\De} \nabla u^\eps\cdot \nabla v\d x
  =\int_{D^\eps} f v\d x +\int_{\partial \De} \partial_n u^\eps v\d\sigma.
\end{align}
Here, $n$ denotes the exterior unit normal field to $\partial \De$, and $\partial_n u^\eps$ denotes the normal derivative of $u^\eps$.
Subtracting \eqref{eq:u_weak} and \eqref{eq:ue_weak}, we obtain
\begin{align}\label{eq:difference}
 \int_{D} \nabla (u-u^\eps) \cdot\nabla v\d x=\int_{\Gamma_\eps} fv\d x -\int_{\partial\De} \partial_n u^\eps v\d \sigma\quad\text{for all }v\in H^1_0(D).
\end{align}

\begin{theorem}\label{thm:errorH1}
 Let $f\in L^2(D)$, and let $u\in H_0^1(D)\cap H^2(D)$ be the solution to \eqref{eq:model1_wlog}, and let $u^\eps\in H^1_0(\De)\cap H^2(\De)$ be the solution to \eqref{eq:model2_wlog}. Then
%  there exists a constant $C>0$ independent of $\eps$ such that 
 \begin{align*}
  \|\nabla u-\nabla u^\eps\|_{L^2(D)}\leq C \sqrt{\eps}  \|f\|_{L^2(D)}.
 \end{align*}
\end{theorem}
\begin{proof}
Using \eqref{eq:trace_of_diffuse_u}, we see that $u-u^\eps \in H^1_0(D)$ is a valid test function for \eqref{eq:difference}, i.e. 
\begin{align}\label{eq:errH1}
 \|\nabla u-\nabla u^\eps\|_{L^2(D)}^2=\int_{\Gamma_\eps} f(u-u^\eps)\d x -\int_{\partial\De} \partial_n u^\eps (u-u^\eps)\d \sigma.
\end{align}
 Using Theorem~\ref{thm:estimate} and \eqref{eq:trace_of_diffuse_u}, we obtain
 \begin{align*}
  \int_{\Gamma_\eps} f (u-u^\eps) \d x\leq  C\eps\|f\|_{L^2(\Gamma_\eps)} \|\nabla u-\nabla u^\eps\|_{L^2(D)}.
 \end{align*}
 Furthermore, an application of the Cauchy-Schwarz inequality yields
 \begin{align*}
  \int_{\partial\De} \partial_n u^\eps (u-u^\eps)\d \sigma &\leq \| \partial_n u^\eps\|_{L^2(\partial\De)} \|u-u^\eps\|_{L^2(\partial\De)}.
 \end{align*}
 The term $\| \partial_n u^\eps\|_{L^2(\partial\De)}$ can be estimated in terms of $\|f\|_{L^2(D)}$ using Lemma~\ref{lem:trace_unweighted} and Lemma~\ref{lem:regularity_diffuse}.
 For the other term, we employ \eqref{eq:trace_of_diffuse_u} and Lemma~\ref{lem:error_perturbed_interface}, which gives
%  Again using \eqref{eq:trace_of_diffuse_u}, we obtain from Lemma~\ref{lem:error_perturbed_interface}
 \begin{align*}
   \|u-u^\eps\|_{L^2(\partial\De)}\leq C \sqrt{\eps} \|\nabla u-\nabla u^\eps\|_{L^2(D)}.
 \end{align*}
 Using these estimates in \eqref{eq:errH1} completes the proof.
%  The interface integrals of $\partial_n u^\eps$ can be estimated using Lemma~\ref{lem:trace_unweighted} and Lemma~\ref{lem:regularity_diffuse}.
\end{proof}
%
% \subsection{$L^2$-estimates}\label{sec:nitsche}
% In this section we prove an $L^2$-error estimate.
% Setting $W=L^2(D)$, $V=H^1_0(D)$ and $W=H^1_0(D^\eps)$ in Lemma~\ref{lem:AubinNitsche}, we obtain 
Using the Aubin-Nitsche lemma \cite{Ciarlet2002}, we obtain
\begin{align}\label{eq:best_approx_dual}
 \|u-u^\eps\|_{L^2(D)}^2 \leq C \|\nabla u-\nabla u^\eps\|_{L^2(D)} \inf_{v\in H^1_0(D^\eps)} \|\nabla z-\nabla v\|_{L^2(D)},
\end{align}
% We start with recalling Galerkin orthogonality, i.e.
% \begin{align}\label{eq:Galerkin_orthogonality}
%  \int_{D^\eps} \nabla (u-u^\eps) \cdot\nabla v\d x = 0\quad\text{for all } v\in H^1_0(D^\eps).
% \end{align}
% % between the solutions $u$ of \eqref{eq:model} and $u^\eps$ of \eqref{eq:Dirichlet_diffuse}.
% % To derive the $L^2$-estimate first observe that
% % \begin{align*}
% %  \| u-u^\eps\|_{L^2(D)}^2 =\| u\|_{L^2(\Gamma_\eps)}^2 + \| u-u^\eps\|_{L^2(D^{\eps})}^2.
% % \end{align*}
% % In view of Theorem~\ref{thm:estimate}, we obtain for the first term
% % \begin{align}
% %  \| u\|_{L^2(\Gamma_\eps)} \leq C\eps \|u\|_{H^1(D)}.
% % \end{align}
% To derive the $L^2$-estimate we employ a standard duality argument.
% The second term can be estimated via a duality argument.
% To do so, we employ an Aubin-Nitsche duality argument. 
where $z\in H^1_0(D)\cap H^2(D)$  denotes the solution to
\begin{align*}%\label{eq:dual}
 -\Delta z = u-u^\eps\quad\text{in }D,
\end{align*}
which exists by Lemma~\ref{lem:regularity_sharp}. % with $f=u-u^\eps$ and $g=0$. 
% Using $u-u^\eps\in H^1_0(D)$ as a test function for \eqref{eq:dual} and Galerkin orthogonality \eqref{eq:Galerkin_orthogonality}, we obtain
% \begin{align}\label{eq:best_approx_dual}
%  \| u-u^\eps\|_{L^2(D)}^2 \leq \inf_{v\in H^1_0(D^\eps)} \|\nabla z -\nabla v\|_{L^2(D)} \|\nabla u-\nabla u^\eps\|_{L^2(D)}.
% \end{align}
Using $v=z^\eps\in H^1_0(D^\eps)\cap H^2(D^\eps)$, defined as the solution to
\begin{align*}%\label{eq:dual2}
 -\Delta z^\eps = u-u^\eps\quad\text{in }D^\eps,
\end{align*}
in \eqref{eq:best_approx_dual}, 
% with $z^\eps $ being the solution to
% which exists by Lemma~\ref{lem:regularity_diffuse},
% and satisfies
% Moreover, we have that 
% \begin{align}\label{eq:dual_estimate}
%  \|z^\eps\|_{H^2(D)} \leq C \|u-u^\eps\|_{L^2(D)}.
% \end{align}
we obtain as a direct consequence of Theorem~\ref{thm:errorH1} the following statement.
% Due to \eqref{eq:trace_of_diffuse_u}, $u-u^\eps\in H^1_0(D)$ is a valid test function for \eqref{eq:dual}.
% Hence, using \eqref{eq:difference}, we obtain
% \begin{align}\label{eq:errorL2}
%  \|u-u^\eps\|_{L^2(D)}^2 &= \int_D\nabla z \cdot\nabla(u-u^\eps)\d x=\int_{\Gamma_\eps} f z\d x -\int_{\partial\De} \partial_n u^\eps z\d \sigma.
% \end{align}   
%
\begin{theorem}\label{thm:L2estimate}
 Let the assumptions of Theorem~\ref{thm:errorH1} hold true. Then
%  there exists a constant $C>0$ such that
 \begin{align*}
  \|u-u^\eps\|_{L^2(D)} \leq C \eps \|f\|_{L^2(D)}.
 \end{align*}
\end{theorem}
% \begin{proof}
%   The proof is similar to that of Theorem~\ref{thm:L2estimate}, i.e. we have to estimate the right-hand side of \eqref{eq:errorL2}.
%   Applying Theorem~\ref{thm:estimate} twice and using \eqref{eq:dual_estimate}, we obtain
%   \begin{align*}
%   \|z\|_{L^2(\Gamma_\eps)}\leq C\eps^{3/2} \|z\|_{H^2(D)} \leq C \eps^{3/2}\|u-u^\eps\|_{L^2(D)}.
%   \end{align*}
%   The first term of the right-hand side of \eqref{eq:errorL2} can thus be estimated as follows
%   \begin{align*}
%   \int_{\Gamma_\eps} f z \d x
%   \leq C\eps^{3/2}  \|f\|_{L^2(\Gamma_\eps)} \|u-u^\eps\|_{L^2(D)}.
%   \end{align*}
%   For the second term, Cauchy-Schwarz inequality, Lemma~\ref{lem:error_perturbed_interface} and Theorem~\ref{thm:estimate} yield
%   \begin{align*}
%    \int_{\partial\De} \partial_n u^\eps z\d x \leq C\sqrt{\eps} \|\partial_n u^\eps\|_{L^2(\partial\De)} \|\partial_n z\|_{L^2(\Gamma_\eps)}
%    \leq C\eps \|\partial_n u^\eps\|_{L^2(\partial\De)} \|z\|_{H^2(D)}.
%   \end{align*}
%   The interface integrals of $\partial_n u^\eps$ can be estimated using Lemma~\ref{lem:trace_unweighted}.
% \end{proof}
%
% A combination of Theorem~\ref{thm:errorH1} and Theorem~\ref{thm:L2estimate} proves Theorem~\ref{thm:main1}.
%
\section{Interface problem involving $\Se_h$}\label{sec:interface}
As a first step towards a practical numerical scheme, we let $\T_h$ be a shape regular triangulation of $\Omega$, cf. \cite[Definition 4.4.13]{BrennerScott}.
% , and $\Se_h=\cup\{T\in\T_h: \Se\cap T\neq \emptyset\}$.
Replacing the domain $\Omega\setminus\Se$ by $\Omega\setminus\Se_h$, where $\Se_h$ is defined in \eqref{eq:Seh}, yields an important change in the geometry.
Firstly, the distance of points in $\partial \Se_h$ to $\Gamma$ is not constant anymore, whereas $\partial\Se=\{|d_\Gamma|=\eps\}$, and
secondly, for fixed $\T_h$, $\partial\Se_h$ is only polygonal. 
Therefore, we cannot rely on $H^2$-regularity of the dual problem to \eqref{eq:dirichlet_diffuse_h}.
%, and the usual regularity results for $C^{1,1}$-domains do not apply.
%
% In order to derive $H^1$ and $L^2$ estimates we will use Ce\'a's Lemma as an essential tool, cf. \cite{Blair1973}.
% \cite{Aubin1967}.
As in the previous section, we consider only the case $g=0$, $D=D_1$ and $\Deh = D_1\setminus \Se_h$ in detail; the remaining cases follow with similar arguments. 
Similar to the previous section, we denote by $u^{\eps+\delta} \in H^1_0(D^{\eps+\delta})\cap H^2(D^{\eps+\delta})$ the solution to \eqref{eq:model2_wlog} with $D^{\eps}$ replaced by $D^{\eps+\delta}$,
where $\delta=\max\{{\rm diam}(T): T\cap \Se\neq \emptyset, T\in\T_h\}$.
% , which by Lemma~\ref{lem:regularity_diffuse} exists.
We consider \eqref{eq:dirichlet_diffuse_h} in weak form, i.e. let $\ueh\in H^1_0(\Deh)$ be the solution to
\begin{align}\label{eq:ueh_weak}
 \int_{\Deh} \nabla\ueh \cdot\nabla v\d x = \int_{\Deh} fv\d x\quad\text{for all } v\in H^1_0(\Deh).
\end{align}
Since $\Deh$ is a non-convex polygonal domain in general, the solution $\ueh$ to \eqref{eq:ueh_weak} is not regular enough to repeat the arguments of the previous section. Our error estimate is based on the following observation
% We observe that
\begin{align}\label{eq:inclusion}
 D^{\eps+\delta} \subset \Deh \subset \De\subset D.
\end{align}
\begin{theorem}\label{thm:errorH1_h}
 Let $f\in L^2(D)$, and let $u^{\eps}\in H^1_0(D)\cap H^2(D)$ be the solution to \eqref{eq:model2_wlog}, and let $\ueh\in H^1_0(\Deh)$ be the solution to \eqref{eq:ueh_weak}. Then there holds
%  there exists a constant $C>0$ independent of $\eps$ and $\delta$ such that 
 \begin{align*}
  \| \nabla u^\eps- \nabla \ueh\|_{L^2(D)}\leq C \sqrt{\delta}  \|f\|_{L^2(D)}
 \end{align*}
 with $\delta=\max\{{\rm diam}(T): T\cap \partial\Se\neq \emptyset,\ T\in \T_h\}$.
\end{theorem}
\begin{proof}
 In view of \eqref{eq:inclusion}, we have that $H^1_0(\Deh)$ is closed in $H^1_0(\De)$, where functions in $H^1_0(\Deh)$ are extended by zero.
 Using Ce\'a's Lemma \cite{BrennerScott} %~\ref{lem:cea} 
 and $u^{\eps+\delta}\in H^1_0(\Deh)$ we obtain
\begin{align*}
 \|\nabla u^\eps-\nabla \ueh\|_{L^2(D)}&\leq \inf_{v\in H^1_0(\Deh)} \|\nabla u^\eps-\nabla v\|_{L^2(D)} \leq \| \nabla u^\eps-\nabla u^{\eps+\delta}\|_{L^2(D)}
 \leq C \sqrt{\delta} \|f\|_{L^2(D)},%\label{eq:cea}
\end{align*}
where we also used Theorem~\ref{thm:errorH1} with $D$ replaced by $\De$ and $u$ replaced by $u^\eps$.
\end{proof}

% An application of the Aubin-Nitsche lemma \cite{Ciarlet2002} yields
% \begin{align}\label{eq:best_approx_dual2}
%  \|u-\ueh\|_{L^2(D)}^2 \leq C \|\nabla u-\nabla \ueh\|_{L^2(D)} \inf_{v\in H^1_0(\Deh)} \|\nabla z^{\eps,h}-\nabla v\|_{L^2(D)},
% \end{align}
% where $z^{\eps,h}\in H^1_0(D)\cap H^2(D)$ is the unique solution to
% \begin{align*}
%  -\Delta z^{\eps,h}= u-\ueh\quad\text{in } D.
% \end{align*}
% The first term on the right-hand side of \eqref{eq:best_approx_dual2} can be estimated using Theorem~\ref{thm:errorH1_h}.
% The second term can be estimated using a similar reasoning as in the proof of Theorem~\ref{thm:errorH1_h}.
% Summarizing, we have the following statement.

An application of the Aubin-Nitsche lemma \cite{Ciarlet2002} yields
\begin{align}\label{eq:best_approx_dual2}
 \|u^\eps-\ueh\|_{L^2(D)}^2 \leq C \|\nabla u^\eps-\nabla \ueh\|_{L^2(D)} \inf_{v\in H^1_0(\Deh)} \|\nabla z^{\eps,h}-\nabla v\|_{L^2(D)},
\end{align}
where $z^{\eps,h}\in H^1_0(\De)\cap H^2(\De)$ is the unique solution to
\begin{align*}
 -\Delta z^{\eps,h}= u^\eps-\ueh\quad\text{in } \De.
\end{align*}
The first term on the right-hand side of \eqref{eq:best_approx_dual2} can be estimated using Theorem~\ref{thm:errorH1_h}.
The second term can be estimated using a similar reasoning as in the proof of Theorem~\ref{thm:errorH1_h}.
Summarizing, we have the following statement.

\begin{theorem}\label{thm:L2estimate_h}
 Let the assumptions of Theorem~\ref{thm:errorH1_h} hold true. Then there holds
%  there exists a constant $C>0$ independent of $\eps$, $h$, and $\delta$ such that
 \begin{align*}
  \|u^\eps-\ueh\|_{L^2(D)} \leq C \delta \|f\|_{L^2(D)}
 \end{align*}
 with $\delta$ as in Theorem~\ref{thm:errorH1_h}.
\end{theorem}
Theorem~\ref{thm:main2} is now a direct consequence of Theorem~\ref{thm:errorH1_h} and Theorem~\ref{thm:L2estimate_h} in combination with Theorem~\ref{thm:errorH1} and Theorem~\ref{thm:L2estimate}, respectively.
Complementing the result in \cite{FranzGaertnerRoosVoigt2012}, we also give an $L^\infty$-error estimate; see  \cite{Thomee1973}.
\begin{theorem}\label{thm:errorlinfty_h}
 Let the assumptions of Theorem~\ref{thm:errorH1_h} hold true, but let additionally $f\in L^p(D)$ for some $p>n$. Then there holds
%  there exists a constant $C>0$ independent of $\eps$, $h$ and $\delta$ such that
 \begin{align*}
  \|u^\eps-\ueh\|_{L^\infty(\Deh)} \leq C \delta \|f\|_{L^p(D)}
 \end{align*}
 with $\delta$ as in Theorem~\ref{thm:errorH1_h}. For the solution $u\in H^1_0(D)\cap H^2(D)$ of \eqref{eq:model1_wlog}, there holds
  \begin{align*}
  \|u^\eps-\ueh\|_{L^\infty(\Deh)} \leq C (\eps+\delta) \|f\|_{L^p(D)}.
 \end{align*}
\end{theorem}
\begin{proof}
 We proceed as in \cite{Thomee1973}.
 We observe that $w=u-\ueh \in H^1(\Deh)$ is the weak solution to 
 \begin{align*}
  -\Delta w = 0\quad\text{in } \Deh,\quad w=u \quad\text{on } \partial \Deh.
 \end{align*}
 Using the weak maximum principle \cite[Theorem~8.1]{GT2001}, we see that
 \begin{align*}
  \|u-\ueh\|_{L^\infty(\Deh)} \leq \|u\|_{L^\infty(\partial \Deh)}.
 \end{align*}
 In view of \eqref{eq:projection} below, every $\bar x\in \partial \Deh$ can be written uniquely as $\bar x = x+d_\Gamma(\bar x) n(x)$, where $x=p_\Gamma(\bar x)\in\partial D$ and $n(x)=\nabla d_\Gamma(x)$ denotes the exterior unit normal to $\partial D$.
 Since $f\in L^p(D)$, we have $u\in W^{2,p}(D)$ \cite{GT2001} and $u\in W^{1,\infty}(D)$ by embedding \cite{Adams1975}.
 Since $u=0$ on $\partial D$, we have 
 \begin{align*}
  |u(\bar x)|&= |\int_0^{d_\Gamma(\bar x)} \partial_n u(x+tn(x)) \d t| \leq (\eps+\delta) \|\partial_n u\|_{L^\infty(D\setminus D^{\eps+\delta})}\leq C (\eps+\delta) \|u\|_{W^{2,p}(D)},
 \end{align*}
 where we used $|d_\Gamma(\bar x)|\leq \eps+\delta$. The other assertion is shown similarly.
\end{proof}

\section{Galerkin approximations}\label{sec:discrete_error}
The error analysis of the Galerkin method defined in \eqref{eq:Galerkin} can be divided in two parts, i.e., similar to the previous sections, it is sufficient to show error estimates in $D_1$ and in $D_2$ separately.
In the following we will again only consider the case $D=D_1$ and $\Deh = D_1\setminus \Se_h$, while the case $D=D_2$ can be treated similarly.
% Using the restriction operator $R_{\Deh}:H^1_0(\Omega)\to H^1(\Deh)$, $R_{\Deh}v=v_{\mid \Deh}$,
% we let $\Wh=R_{\Deh}\Vh$, with 
As in the introduction, we let $\Vh=P_1(\T_h)\cap H^1_0(\Omega)$ be the finite element space of continuous piecewise affine functions associated to $\T_h$.
Furthermore, we denote by $I_h:C^0(\overline{\Omega})\to P_1(\T_h)\cap H^1(\Omega)$ the nodal interpolation operator. 
By construction of $\Deh$, we see that $\{T\in \T_h: T\cap \Deh\neq \emptyset\}$ is a shape regular triangulation of $\Deh$.
The corresponding finite element space is obtained by restriction, i.e. $\Wh=\{v_{h\mid \Deh}: v_h\in \Vh\}$.
Since $g\in C^0(\overline{\Omega})$, we can define
% For the incorporation of boundary values we denote by $I_h g \in\Wh$ the usual nodal interpolant of $g$, and define
\begin{align}
 \Wehg=\{v_h \in \Wh: v_{h}=I_h g \text{ on } \partial\Deh\}.
\end{align}
Compatible with \eqref{eq:Galerkin}, we define $u^\eps_h\in \Wehg$ by
\begin{align}
 \int_{\Deh}\nabla u^\eps_h \cdot\nabla v_h \d x = \int_{\Deh} f v_h \d x\quad\text{for all } v_h\in \Wehn,\label{eq:Galerkin2}
\end{align}
and set $u^\eps_h=I_h g$ on $D\setminus\Deh$.
Furthermore, we let
$\kappa=\max\{ {\rm diam}(T):T\cap\partial\mathcal{S}^{\eps+\delta}\neq\emptyset,\ T\in \T_h\}$ and $h=\max\{{\rm diam}(T):\ T\in \T_h\}$.
We then have the following statement.
\begin{theorem}\label{thm:error_H1_Galerkin}
 Let $\ueh\in H^1(\Deh)$ be the solution to \eqref{eq:ueh_weak} such that $\ueh=g$ on $\partial\Deh$, and let $u_h^\eps \in \Wehg$ be the solution to \eqref{eq:Galerkin2}.
 Then 
%  there exists a constant $C>0$ independent of $h$, $\eps$, $\kappa$ and $\delta$ such that
 \begin{align*}
  \| \ueh-u_h^\eps\|_{H^1(D)} \leq C (\sqrt{\delta}+\kappa^{\frac{2}{3}}+h)(\|f\|_{L^2(D)} + \|g\|_{H^2(D)}).
\end{align*}
\end{theorem}
\begin{proof}
  We observe that $u_h^\eps-v_h\in \Wehn\subset H^1_0(\Deh)$ for any $v_h\in\Wehg$.
  Therefore, Ce\'a's Lemma \cite{BrennerScott} provides the following quasi-optimal approximation result 
  \begin{align}
    \| \ueh-u_h^\eps\|_{H^1(\Deh)} &\leq C\inf_{v\in \Wehg} \|\ueh - v_h\|_{H^1(\Deh)}.\label{eq:Cea_h}
  \end{align}
%  In view of \eqref{eq:Cea_h}, 
 We will estimate the best-approximation error in the following.
 Let $u^\eps\in H^2(\De)$ be the solution to \eqref{eq:model2_wlog} with Dirichlet boundary datum $g$ on $\partial \De$, and let $u^{\eps+\delta}\in H^2(D^{\eps+\delta})$ be the solution to \eqref{eq:model2_wlog} with $\De$ replaced by $D^{\eps+\delta}$ and Dirichlet boundary datum $g$ on $\partial D^{\eps+\delta}$.
 In view of \eqref{eq:Dirichlet_diffuse}, we set $u^{\eps+\delta}=g$ on $D\setminus D^{\eps+\delta}$, which implies $u^{\eps+\delta}\in H^1(D)$.
 Using Theorem~\ref{thm:errorH1}, Theorem~\ref{thm:errorH1_h} with $\tilde f= f+\Delta g$, we obtain
 \begin{align*}
  \inf_{v_h\in \Wehg} \| \ueh - v_h\|_{H^1(\Deh)}&\leq \inf_{v\in \Wehg} \|u^{\eps+\delta} - v_h\|_{H^1(\Deh)} +  \|\ueh - u^\eps\|_{H^1(\Deh)}+\|u^\eps-u^{\eps+\delta}\|_{H^1(\Deh)} \\
  &\leq \inf_{v_h\in \Wehg} \|u^{\eps+\delta} - v_h\|_{H^1(\Deh)} + C\sqrt{\delta}(\|f\|_{L^2(D)}+\|g\|_{H^2(D)}).
 \end{align*}
 By embedding \cite{Adams1975}, we have that $u^{\eps+\delta}\in C^0(\overline{D})$. Therefore, $I_h u^{\eps+\delta}\in \Wehg$ is well-defined and
%  the nodal pointwise interpolant $I_h u^{\eps+\delta}\in \Wehg$ is well-defined, and
 \begin{align*}
  \inf_{v_h\in \Wehg} \|u^{\eps+\delta} - v_h\|_{H^1(\Deh)}\leq \|u^{\eps+\delta} - I_h u^{\eps+\delta}\|_{H^1(\Deh)}.
 \end{align*}
 Next, we estimate the right-hand side of the latter estimate on each element $T\in\T_h$. 
 We distinguish three cases.
 
 (i) $T\cap D^{\eps+\delta}=\emptyset$. %Then, $u^{\eps+\delta}_{\mid T} -(I_h u^{\eps+\delta})_{\mid T} = (g - I_h g)_{\mid T}$.
 Since $u^{\eps+\delta}_{\mid T}=g_{\mid T}\in H^2(T)$, standard interpolation error estimates \cite[Theorem 4.4.4]{BrennerScott} yield
 $$
 \|u^{\eps+\delta} - I_h u^{\eps+\delta}\|_{H^1(T)}=\|g - I_h g\|_{H^1(T)}\leq C {\rm diam}(T)\|u\|_{H^2(T)}.
 $$
 
 (ii) $T\subset D^{\eps+\delta}$. Then $u^{\eps+\delta}_{\mid T}\in H^2(T)$, and as in (i) we have
 $$
 \|u^{\eps+\delta} - I_h u^{\eps+\delta}\|_{H^1(T)}\leq C {\rm diam}(T)\|u^{\eps+\delta}\|_{H^2(T)}.
 $$
 
 (iii) $int(T)\cap \partial D^{\eps+\delta}\neq \emptyset$. By embedding \cite[Theorem 5.4]{Adams1975}, we have that $u^{\eps+\delta}\in W^{1,p}(D^{\eps+\delta})$ for $p<\infty$ if $n=2$ or $p=6$ if $n=3$. Since $u^{\eps+\delta}=g$ on $\partial D^{\eps+\delta}$, we further have that $u^{\eps+\delta}\in W^{1,p}(T)$. Using H\"older's inequality and \cite[Theorem 4.4.4]{BrennerScott} we thus obtain
 \begin{align*}
  \|u^{\eps+\delta} - I_h u^{\eps+\delta}\|_{H^1(T)}&\leq |T|^{\frac{1}{2}-\frac{1}{p}} \|u^{\eps+\delta} - I_h u^{\eps+\delta}\|_{W^{1,p}(T)}
  \leq C |T|^{\frac{1}{3}} \|u^{\eps+\delta}\|_{W^{1,p}(T)}\\
  &\leq C |T|^{\frac{1}{3}} (\|g\|_{H^2(T\setminus D^{\eps+\delta})}+\|u^{\eps+\delta}\|_{H^2(T\cap D^{\eps+\delta})}).
 \end{align*}
 In cases (i) and (ii) we have that ${\rm diam}(T)\leq h$, and in case (iii) we have that $|T|\leq {\rm diam}(T)^2 \leq  \kappa^2$.
 % $T$ is in the ``next'' layer.
 Therefore, collecting the estimates in (i), (ii) and (iii), we obtain
 \begin{align*}
  &\|u^{\eps+\delta} - I_h u^{\eps+\delta}\|_{H^1(\Deh)}^2 = \sum_{T\subset \overline{\Deh}} \|u^{\eps+\delta} - I_h u^{\eps+\delta}\|_{H^1(T)}^2\\
  &=\sum_{int(T)\cap  \partial D^{\eps+\delta}=\emptyset} \|u^{\eps+\delta} - I_h u^{\eps+\delta}\|_{H^1(T)}^2
%   +\sum_{T\subset D^{\eps+\delta}} \|u^{\eps+\delta} - I_h u^{\eps+\delta}\|_{H^1(T)}^2
  + \sum_{int(T)\cap\partial D^{\eps+\delta}\neq\emptyset} \|u^{\eps+\delta} - I_h u^{\eps+\delta}\|_{H^1(T)}^2\\
  &\leq C \sum_{T\subset D^{\eps+\delta}} h^2 \|u^{\eps+\delta}\|_{H^2(T)}^2 + C\sum_{ int(T)\cap\partial D^{\eps+\delta}\neq\emptyset} \kappa^{\frac{4}{3}} (\|g\|_{H^2(T\setminus D^{\eps+\delta})}^2+\|u^{\eps+\delta}\|_{H^2(T\cap D^{\eps+\delta})}^2)\\
  &\leq C (h^2 +\kappa^{\frac{4}{3}})  \|u^{\eps+\delta}\|_{H^2(D^{\eps+\delta})}^2 + C \kappa^{\frac{4}{3}} \|g\|_{H^2(\Deh\setminus D^{\eps+\delta})}^2.
 \end{align*}
%  Since, we may choose $p=6$ for $n\in \{2,3\}$, we obtain $\kappa^{1-\frac{2}{p}} \leq \kappa^{\frac{2}{3}}$. %\leq \sqrt{\delta}$.
 Since $\|u^{\eps+\delta}\|_{H^2(D^{\eps+\delta})}\leq C (\|f\|_{L^2(D)}+\|g\|_{H^2(D)})$ by Lemma~\ref{lem:regularity_diffuse}, and $u_h^\eps-\ueh=I_h g-g$ on $D\setminus \Deh$, which can be estimated as above, the proof is complete.
\end{proof}
\begin{remark}
 According to \cite{Grisvard85} we have $\ueh\in H^{3/2+s_0}(\Deh)\cap H^1_0(\Deh)$ with $s_0>0$ depending on the shape regularity of $\T_h$. The above error estimates may alternatively be derived by estimating %\eqref{eq:cea1}
 $\inf_{v\in \Wehg} \|\ueh-v\|_{H^1(\Deh)}$
 using interpolation, cf. \cite{BrennerScott}. However, the constants will depend on $\|\ueh\|_{H^{3/2}(\Deh)}$, and as the number of re-entrant corners in $\Deh$ might in general be unbounded as $h\to 0$ the corresponding estimates might not be uniform in $h$ anymore.
\end{remark}
$L^2$-estimates are derived again via a duality argument. Note that the Dirichlet problem on $\Deh$ is not $H^2$-regular, and thus, we cannot rely on standard estimates as in the previous sections.
Duality arguments for the approximation of inhomogeneous Dirichlet boundary value problems for $H^2$-regular problems have e.g. been investigated in \cite{BaCaDo2004}.
\begin{theorem}\label{thm:error_L2_Galerkin}
 Let $\ueh\in H^1(\Deh)$ be the solution to \eqref{eq:ueh_weak} with $\ueh=g$ on $\partial\Deh$, and let $u_h^\eps \in \Wehg$ be the solution to \eqref{eq:Galerkin2}. Then 
%  there exists a constant $C>0$ independent of $h$, $\eps$, $\kappa$ and $\delta$ such that
 \begin{align*}
 \| \ueh-u_h^\eps\|_{L^2(D)} \leq C (\delta+\kappa^{\frac{4}{3}}+h^2)(\|f\|_{L^2(D)} + \|g\|_{H^2(D)}).
\end{align*}
\end{theorem}
\begin{proof}
We define $\eh=\ueh-u^\eps_h$, and let $z^{\eps,h}\in H^1_0(\De)\cap H^2(\De)$ be the solution to
\begin{align*}
 -\Delta  z^{\eps,h} = \eh \quad\text{in } \De.
\end{align*}
Then, we obtain upon integration by parts
\begin{align}\label{eq:dual_h}
 \|\eh \|_{L^2(\De)}^2 = \int_{\De} \nabla z^{\eps,h} \nabla \eh \d x - \int_{\partial\De}\partial_n z^{\eps,h}\eh\d\sigma.
\end{align}
We estimate the terms on the right-hand side separately.

(i) Using Galerkin orthogonality, we obtain %for any $v_h\in \Wehn\subset H^1_0(\Deh)$
\begin{align*}
 &\int_{\De} \nabla z^{\eps,h} \nabla \eh\d x = 
 \int_{\De\setminus\Deh} \nabla z^{\eps,h} \nabla \eh\d x + \int_{\Deh} \nabla z^{\eps,h} \nabla \eh\d x\\
 &\leq \|\nabla z^{\eps,h}\|_{L^2(\De\setminus\Deh)} \|\nabla \eh\|_{L^2(\De\setminus\Deh)} + \| \nabla \eh\|_{L^2(\Deh)} \inf_{v_h\in\Wehn}\|\nabla z^{\eps,h}-\nabla v_h\|_{L^2(\Deh)}.
\end{align*}

(ia) Using \eqref{eq:inclusion}, Theorem~\ref{thm:estimate}, Lemma~\ref{lem:trace_unweighted} and Lemma~\ref{lem:regularity_diffuse}, we see that
\begin{align*}
  \|\nabla z^{\eps,h}\|_{L^2(\De\setminus\Deh)}\leq C\sqrt{\delta}\|z^{\eps,h}\|_{H^2(\De)}\leq C \sqrt{\delta}\|\eh\|_{L^2(\De)}.
\end{align*}
Defining the set 
\begin{align*}
 \Leh = \bigcup_{T\in\T_h: T\cap\partial\Se \neq \emptyset} T,
\end{align*}
we see that $\De\setminus\Deh \subset \Leh$. 
Therefore, since $\eh=I_h g-g$ on $\De\setminus \Deh$, we obtain using standard interpolation estimates and the definition of $\delta$ that
\begin{align*}
  \|\nabla \eh\|_{L^2(\De\setminus\Deh)} \leq  \|\nabla I_h g - \nabla g\|_{L^2(\Leh)}\leq C\delta \|g\|_{H^2(\Leh)},
\end{align*}

(ib) The term $ \| \nabla \eh\|_{L^2(\Deh)}$ can be estimated by Theorem~\ref{thm:error_H1_Galerkin}
% we have that
% \begin{align*}
%  \| \nabla \eh\|_{L^2(\Deh)} \leq C (\sqrt{\delta}+\kappa^{\frac{2}{3}}+h)(\|f\|_{L^2(D)} + \|g\|_{H^2(D)}).
% \end{align*}
% In order to estimate the best-approximation error we define 
%  $z^{\eps+\delta}\in H^1_0(D^{\eps+\delta})\cap H^2(D^{\eps+\delta})$ as the solution to
% \begin{align*}
%  -\Delta z^{\eps+\delta}= \eh \quad\text{in } D^{\eps+\delta}.
% \end{align*}
% Hence, we have that
% \begin{align*}
%   \inf_{v_h\in\Wehn}\|\nabla z^{\eps,h}-\nabla v_h\|_{L^2(\Deh)}\leq \|\nabla z^{\eps,h}-\nabla z^{\eps+\delta}\|_{L^2(\Deh)} +\inf_{v_h\in\Wehn}\|\nabla z^{\eps+\delta}-\nabla v_h\|_{L^2(\Deh)}.
% \end{align*}
% For the first term on the right-hand side we use Theorem~\ref{thm:errorH1}, and the second term is estimated similarly as in the proof of Theorem~\ref{thm:error_H1_Galerkin}, i.e.
The best-approximation error can be estimated as in the proof of Theorem~\ref{thm:error_H1_Galerkin}, i.e.
\begin{align*}
  \inf_{v_h\in\Wehn}\|\nabla z^{\eps,h}-\nabla v_h\|_{L^2(\Deh)}\leq C \sqrt{\delta}\|\eh\|_{L^2(D)} + C (\sqrt{\delta}+\kappa^{\frac{2}{3}}+h)\|\eh\|_{L^2(D)}.
\end{align*}
Summarizing, for the first term on the right-hand side in \eqref{eq:dual_h} we have
\begin{align}\label{eq:dual_h1}
  \int_{\De} \nabla z^{\eps,h} \nabla \eh \d x \leq C(\sqrt{\delta}+\kappa^{\frac{2}{3}}+h)^2(\|f\|_{L^2(D)} + \|g\|_{H^2(D)})\|\eh\|_{L^2(D)}.
\end{align}

(ii) An application of the Cauchy-Schwarz inequality and $H^2$-regularity of $z^{\eps,h}$ yield
\begin{align*}
 \int_{\partial\De}\partial_n z^{\eps,h}\eh\d\sigma \leq \|\partial_n z^{\eps,h}\|_{L^2(\partial \De)} \|\eh\|_{L^2(\partial\De)} \leq C\|\eh\|_{L^2(\De)} \|\eh\|_{L^2(\partial\De)},
\end{align*}
where we have also applied Lemma~\ref{lem:trace_unweighted}.
We now use $\eh=I_h g-g$ on $\Leh$ and $\partial\De\subset \Leh$ to estimate further
\begin{align*}
  \|\eh\|_{L^2(\partial\De)}^2 &= \sum_{T\subset\Leh} \|I_h g -g\|_{L^2(\partial\De\cap T)}^2\\
  &\leq \sum_{T\subset\Leh} |\partial\De\cap T| \|I_h g-g\|_{L^\infty(T)}^2\\
  &\leq  \sum_{T\subset\Leh} |\partial\De\cap T| {\rm diam}(T) \|g\|_{H^2(T)}^2\\
  &\leq C \delta^2 \|g\|_{H^2(\Leh)}^2.
\end{align*}
Here, we used that $|\partial\De\cap T|\leq C{\rm diam}(T)$ and ${\rm diam}(T)\leq \delta$ for all $T$ such that $T\cap \Leh\neq \emptyset$, and \cite[Corollary~4.4.7]{BrennerScott} in order to estimate the $L^\infty$ interpolation error.

Summarizing, for the second term in \eqref{eq:dual_h} we have
\begin{align}\label{eq:dual_h2}
  \int_{\partial\De}\partial_n z^{\eps,h}\eh\d\sigma\leq C \delta \|g\|_{H^2(D)}\|\eh\|_{L^2(D)}.
\end{align}
The proof is finished by observing that
\begin{align*}
 \|\eh\|_{L^2(D)}\leq  \|\eh\|_{L^2(\De)} + \|\eh\|_{L^2(D\setminus\De)} \leq  \|\eh\|_{L^2(\De)} + C h^2 \|g\|_{H^2(D)},
\end{align*}
and using \eqref{eq:dual_h}--\eqref{eq:dual_h2} and Young's inequality to estimate the first term on the right-hand side of the latter inequality.
\end{proof}
A combination of Theorem~\ref{thm:main2} and Theorems~\ref{thm:error_H1_Galerkin}, \ref{thm:error_L2_Galerkin} proves Theorem~\ref{thm:main3}.
\begin{remark}\label{rem:saddle-point}
  Instead of incorporating the condition $\ueh=g$ on $\Se_h$ in the approximation space, this condition might also be incorporated via a saddle-point formulation.
%   Approximations are obtained as solutions of the following problem:
  For instance, one might consider:
  Find $(u_h^\eps,\lambda_h^\eps) \in \Vh\times X_h^\eps$ such that
  \begin{align}
    \int_{\Omega} \nabla u^\eps_{h}\cdot \nabla v_h\d x + \int_{\Se_h} \lambda^\eps_{h} v_h\d x &= \int_{\Omega} fv_h \d x, \label{eq:saddle_diffuse1_h}\\
    \int_{\Se_h} u^\eps_{h} \mu_h\d x&=\int_{\Se_h} g \mu_h\d x\label{eq:saddle_diffuse2_h}
  \end{align}
  for all $(v_h,\mu_h) \in \Vh\times X_h^\eps$ with $\Vh$ as above and $X_h^\eps$ chosen such that an inf-sup condition holds, i.e. there exists $c>0$ such that for all $\mu_h\in X_h^\eps$
  \begin{align}\label{eq:infsup}
    \int_{\Se_h} \mu_h v_h\d x \geq c \|v_h\|_{H^1(\Omega)}\|\mu_h\|_{X_h^\eps}\quad\text{for all } v_h\in \Vh
  \end{align}
  holds. On the continuous level such a condition may be verified for functions $\mu\in X$ being defined as the closure of $L^2(\Se_h)$ with respect to the norm
  $$
  \|\mu\|_{X} = \sup_{v\in H_0^1(\Omega)\setminus 0} \int_{\Se_h} \mu v\d x/\|v\|_{H^1(\Omega)}.
  $$
  The verification of \eqref{eq:infsup} is completed by constructing a Fortin projector $P:H^1(\Se_h)\to \Vh$ such that $P$ is bounded and
  \begin{align*}
   \int_{\Se_h} \mu_h (Pv-v)\d x=0 \quad\text{for all } v\in H^1(\Se_h)\text{ and } \mu_h\in X_h,
  \end{align*}
  see \cite[Proposition 2.8]{BrezziFortin}.
  If $X_h=\{v_{h\mid \Se_h}: v_h\in \Vh\}$, the construction of a Fortin projector amounts to $H^1$-stability of the $L^2$-projection, which has recently been shown in \cite{BankYserentant2014} for a large class of finite element approximation spaces.
  We remark that, if $g$ is replaced by $I_h g$ in \eqref{eq:saddle_diffuse2_h}, then the choice $\mu_h=u^\eps_h-I_hg$ yields $u^\eps_h=I_h g$. Choosing $v_h\in\Vehn$ then shows that $u^\eps_h$ is a solution to \eqref{eq:Galerkin}.
%   ; let us also mention   \cite{BankDupont1981,Yserentant1990,BrambleXu1991}.
%   One advantage of the saddle-point formulation is stability with respect to perturbations of the right-hand side. For instance, variational crimes arising from inexact integration of $\int_{\Se_h} g v\d x$ can be estimated easily.
  Penalized saddle-point problems might also be considered, cf. \cite[II.4]{BrezziFortin}, and, for instance, the method in \cite{ReuterHillHarrison2012} might be interpreted as such.
  Moreover, surface PDEs, when appropriately extended to $\Se$, might be incorporated as an additional constraint.

\end{remark}

\section{Numerical Examples}\label{sec:numerics}
Demonstrating the validity of the above derived results, we set $\Omega=(-2,2)\times (-2,2)\subset\RR^2$,
%
% \subsection{Example A}\label{sec:example1}
and let $\Gamma=\partial B_{1}(0)$. Furthermore, we let $x=(x_1,x_2)\in \Omega$ and
\begin{align*}
 u(x)&= (4-x_1^2)(4-x_2^2) \chi_{D_2}(x) + (4-x_1^2)(4-x_2^2) \exp(1-|x|^2) \chi_{\overline{D_1}}(x),\\
 g(x)&= (4-x_1^2)(4-x_2^2) \cos(1-|x|^2),\\
 f(x)&= -\Delta u(x), \quad x\in \Omega\setminus\Gamma,\qquad f(x)=0,\quad x\in\Gamma. %(16-2|x|^2)\chi_{D_2}(x)+ 2\exp(1-|x|^2)(40-57|x|^2+26 x_1^2x_2^2+8x_1^4-2x_1^4 x_2^2+8x_2^4-2x_2^4 x_1^2) \chi_{D_1}(x). 
\end{align*}
Here, $\chi_{\overline{D_1}}$ and $\chi_{D_2}$ denote the characteristic functions of $\overline{D_1}$ and $D_2$, respectively.
One verifies that this choice of functions yields a solution to \eqref{eq:model}.
For the sake of simplicity, we have chosen an arbitrary globally defined and sufficiently smooth function $g$ such that $g(x)=u(x)$ for $x\in \Gamma$. 
% Notice, however, that neither $g_{\mid D_1}=u_{\mid D_1}$ nor $g_{\mid D_2}=u_{\mid D_2}$. 
As noticed in the introduction, if $g$ is defined on $\Gamma$ only, we have to construct a suitable extension of $g$ to a neighborhood of $\Gamma$.
For our numerical tests, we employed the linear nodal interpolant $I_h d_\Gamma$ of $d_\Gamma(x)=|x|-1$ in order to define $\Se_h$, namely, 
setting $\omega_h^\eps(x) = \frac{1}{2}(1+S(-\frac{I_hd_\Gamma(x)}{\eps}))$, we note that ${\rm supp}|\nabla\omega_h^\eps|=\Se_h$.
In all our tests below, we do not employ aligned meshes, i.e. the interface approximations $\partial\Se_h$ are not aligned with $\Gamma$.
In Figure~\ref{fig:sketch} we have depicted the geometric setup.
The Galerkin approximation $u^\eps_h$ is then computed via \eqref{eq:Galerkin} with $f$ replaced by $I_h f$, which is a sufficient approximation for $f$ on $\Omega\setminus\Se_h$ for the chosen $f$.

\begin{figure}
 \centering
 \includegraphics[width=.45\textwidth]{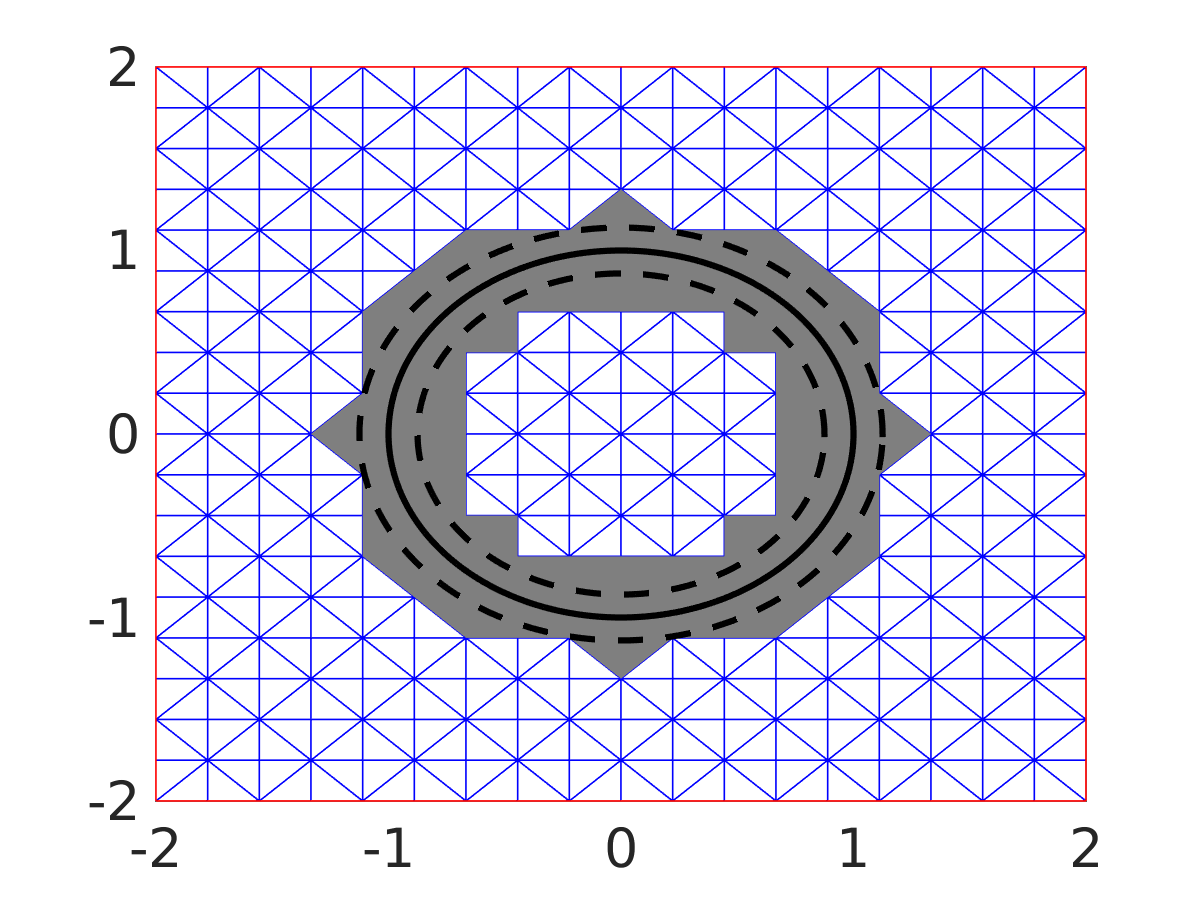}
 \includegraphics[width=.45\textwidth]{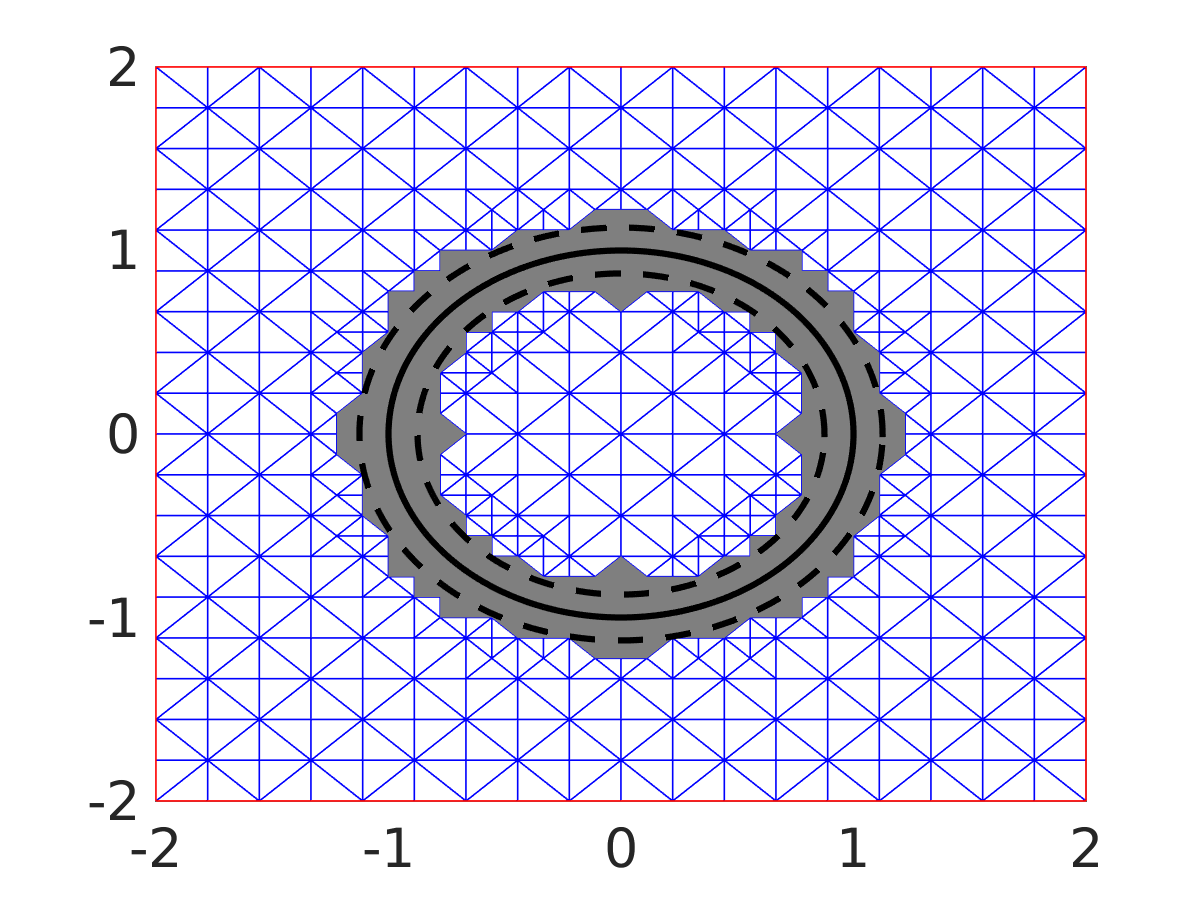}
 \caption{\label{fig:sketch} Sketch of computational domain and corresponding triangulation. The solid black line represents $\Gamma$. The dotted black lines correspond to $\partial\Se$. The gray area corresponds to $\Se_h$ for $\eps=1/8$. Left: $h=\delta$ and $361$ vertices in total. Right: $\delta=h^2$ and $593$ vertices in total.}
\end{figure}

In our first experiment we choose a uniform triangulation, i.e. $\delta=h=\kappa$, with $332\,929$ vertices; cf. Figure~\ref{fig:sketch} left. 
The convergence results for different values of $\eps=1/2^i$, $i=1,\ldots,20$ are depicted in Figure~\ref{fig:convergence_1}. We observe the predicted convergence rates $O(\eps)$ for the $L^p(\Omega)$-norm, $p\in\{2,\infty\}$, and $O(\sqrt{\eps})$ for $H^1(\Omega)$-norm.
In particular, the error behaves monotonically and saturates at a certain level, which is due to the chosen mesh.

\pgfplotstableread{rates_in_eps.dat}{\ratesineps}
\begin{figure}[htb]
        \centering
        \begin{tikzpicture}%[width=.3\textwidth]
                \begin{loglogaxis}[
                        xlabel=$\eps$,
%                         ylabel=$\log_{10}{\Phi}$,
                        width=5.5cm,
                        log basis x=2,
                        log basis y=2,
                        legend pos=north west,
                        ymin=0.125,
                        ]
                        \addplot+[color=black,mark=]              table[header=true,x index=0,y expr=(\thisrowno{0})*50,skip coords between index={7}{20}]{\ratesineps};
                        \addplot+[only marks,mark=x,color=black]  table[header=true,x index=0,y=l2]{\ratesineps};
                        \legend{$O(\eps)$,$L^2$}
                \end{loglogaxis}
        \end{tikzpicture}
        \begin{tikzpicture}%[width=.3\textwidth]
                \begin{loglogaxis}[
                        xlabel=$\eps$,
%                         ylabel=$\log_{10}{\Phi}$,
                        width=5.5cm,
                        log basis x=2,
                        log basis y=2,
                        legend pos=north west
                        ]
                        \addplot+[color=black,mark=]              table[header=true,x index=0,y expr=sqrt(\thisrow{eps})*70,skip coords between index={7}{20}]{\ratesineps};
                        \addplot+[only marks,mark=x,color=black]  table[header=true,x index=0,y=h1]{\ratesineps};
                        \legend{$O(\sqrt{\eps})$,$H^1$}
                \end{loglogaxis}
        \end{tikzpicture}
        \begin{tikzpicture}%[width=.3\textwidth]
                \begin{loglogaxis}[
                        xlabel=$\eps$,
%                         ylabel=$\log_{10}{\Phi}$,
                        width=5.5cm,
                        log basis x=2,
                        log basis y=2,
                        legend pos=north west
                        ]
                        \addplot[color=black,mark=]              table[header=true,x index=0,y expr=(\thisrow{eps})*30,skip coords between index={7}{20}]{\ratesineps};
                        \addplot[only marks,mark=x,color=black]  table[header=true,x index=0,y=linf]{\ratesineps};
                        \legend{$O(\eps)$,$L^\infty$}
                \end{loglogaxis}
        \end{tikzpicture}
      \caption{\label{fig:convergence_1} Convergence rates in $\eps$: Uniform mesh with $332\,929$ vertices. The solid lines correspond to the predicted rates ($\eps$, $\sqrt{\eps}$, $\eps$ from left to right). The actual errors $\|u-u^\eps_h\|_{L^2(\Omega)}$, $\|u-u^\eps_h\|_{H^1(\Omega)}$, $\|u-u^\eps_h\|_{L^\infty(\Omega)}$ (from left to right) are marked by crosses.}
\end{figure}
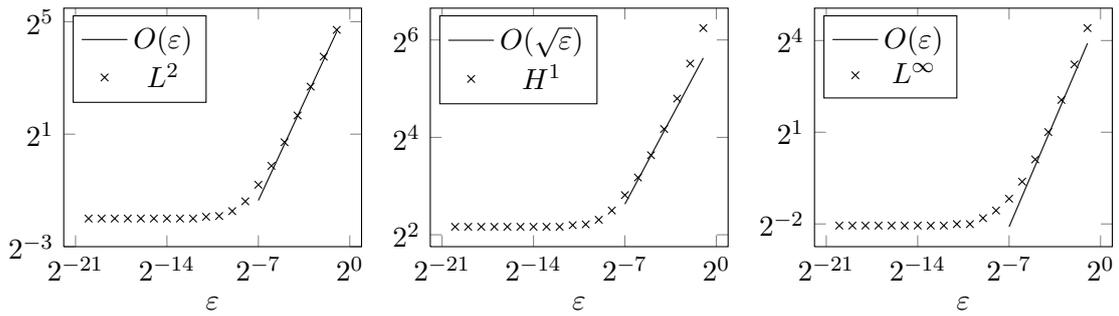

% \begin{figure}
%  \centering
%  \includegraphics[width=.32\textwidth]{l2_eps}
%  \includegraphics[width=.32\textwidth]{h1_eps}
%  \includegraphics[width=.32\textwidth]{linf_eps}
%  \caption{\label{fig:convergence_1} Convergence rates in $\eps$: Uniform mesh with $332\,929$ vertices. The solid lines correspond to the predicted rates ($\eps$, $\sqrt{\eps}$, $\eps$ from left to right). The dotted lines are the actual errors $\|u-u^\eps_h\|_{L^2(\Omega)}$, $\|u-u^\eps_h\|_{H^1(\Omega)}$, $\|u-u^\eps_h\|_{L^\infty(\Omega)}$ (from left to right).}
% \end{figure}

In a second experiment we also chose uniform triangulations, i.e. $\delta=h=\kappa$, and $\eps=1/2^{20}$ fixed. We used different mesh sizes with number of vertices in $\{5\,329, 21\,025, 83\,521, 332\,929\}$. Notice that in this example $h>10^{-3}$ and $\eps\approx 10^{-6}$, i.e. $\eps \ll h$. The convergence results for different mesh sizes are depicted in Figure~\ref{fig:convergence_2}. We observe the predicted convergence rates $O(h)$ for the $L^p(\Omega)$-norm, $p\in\{2,\infty\}$, and $O(\sqrt{h})$ for the $H^1(\Omega)$-norm.

\pgfplotstableread{rates_in_h.dat}{\ratesinh}
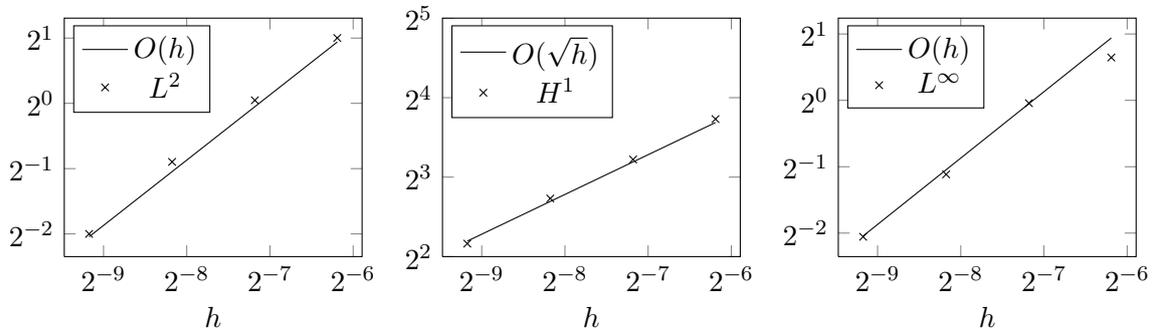
\begin{figure}[htb]
        \centering
        \begin{tikzpicture}%[width=.3\textwidth]
                \begin{loglogaxis}[
                        xlabel=$h$,
%                         ylabel=$\log_{10}{\Phi}$,
                        width=5.5cm,
                        log basis x=2,
                        log basis y=2,
                        legend pos=north west
                        ]
                        \addplot+[color=black,mark=]              table[header=true,x expr=1/(sqrt(\thisrowno{0})),y expr=1/sqrt(\thisrowno{0})*140]{\ratesinh};
                        \addplot+[only marks,mark=x,color=black]  table[header=true,x expr=1/(sqrt(\thisrowno{0})),y index=1]{\ratesinh};
                        \legend{$O(h)$,$L^2$}
                \end{loglogaxis}
        \end{tikzpicture}
        \begin{tikzpicture}%[width=.3\textwidth]
                \begin{loglogaxis}[
                        xlabel=$h$,
%                         ylabel=$\log_{10}{\Phi}$,
                        width=5.5cm,
                        log basis x=2,
                        log basis y=2,
                        ymin=4,
                        ymax=32,
                        legend pos=north west
                        ]
                        \addplot+[color=black,mark=]              table[header=true,x expr=1/(sqrt(\thisrowno{0})),y expr=sqrt(1/sqrt(\thisrowno{0}))*110]{\ratesinh};
                        \addplot+[only marks,mark=x,color=black]  table[header=true,x expr=1/(sqrt(\thisrowno{0})),y index=2]{\ratesinh};
                        \legend{$O(\sqrt{h})$,$H^1$}
                \end{loglogaxis}
        \end{tikzpicture}
        \begin{tikzpicture}%[width=.3\textwidth]
                \begin{loglogaxis}[
                        xlabel=$h$,
%                         ylabel=$\log_{10}{\Phi}$,
                        width=5.5cm,
                        log basis x=2,
                        log basis y=2,
                        legend pos=north west
                        ]
                        \addplot+[color=black,mark=]              table[header=true,x expr=1/(sqrt(\thisrowno{0})),y expr=1/sqrt(\thisrowno{0})*140]{\ratesinh};
                        \addplot+[only marks,mark=x,color=black]  table[header=true,x expr=1/(sqrt(\thisrowno{0})),y index=3]{\ratesinh};
                        \legend{$O(h)$,$L^\infty$}
                \end{loglogaxis}
        \end{tikzpicture}
     \caption{\label{fig:convergence_2} Convergence rates in $h$: Uniform mesh with number of vertices in $\{5\,329, 21\,025, 83\,521, 332\,929\}$ and fixed $\eps=1/2^{20}$. The solid lines correspond to the predicted rates ($h$, $\sqrt{h}$, $h$ from left to right). The actual errors $\|u-u^\eps_h\|_{L^2(\Omega)}$, $\|u-u^\eps_h\|_{H^1(\Omega)}$, $\|u-u^\eps_h\|_{L^\infty(\Omega)}$ (from left to right)  are marked by crosses.}
\end{figure}
% \begin{figure}
%  \centering
%  \includegraphics[width=.32\textwidth]{l2_h}
%  \includegraphics[width=.32\textwidth]{h1_h}
%  \includegraphics[width=.32\textwidth]{linf_h}
%  \caption{\label{fig:convergence_2} Convergence rates in $h$: Uniform mesh with number of vertices in $\{5\,329, 21\,025, 83\,521, 332\,929\}$ and fixed $\eps=1/2^{20}$. The solid lines correspond to the predicted rates ($h$, $\sqrt{h}$, $h$ from left to right). The dotted lines are the actual errors $\|u-u^\eps_h\|_{L^2(\Omega)}$, $\|u-u^\eps_h\|_{H^1(\Omega)}$, $\|u-u^\eps_h\|_{L^\infty(\Omega)}$ (from left to right).}
% \end{figure}

In a third experiment we chose a locally refined mesh such that $\delta=h^2$ and $\kappa\leq 4\delta$, which has been obtained from the meshes in the second experiment by repeated refinement of those triangles having nonempty intersection with $\partial\Se$; cf. Figure~\ref{fig:sketch} right. The resulting meshes have number of vertices in $\{10\,053,$ $41\,445, 168\,717, 681\,957\}$. The diffuse interface width $\eps=1/2^{20}$ is fixed. The convergence results for different mesh sizes are depicted in Figure~\ref{fig:convergence_3}. We observe the predicted convergence rates $O(h^2)$ for the $L^p(\Omega)$-norm, $p\in\{2,\infty\}$, and $O(h)$ for the $H^1(\Omega)$-norm.

\pgfplotstableread{rates_in_h_loc.dat}{\ratesinhloc}
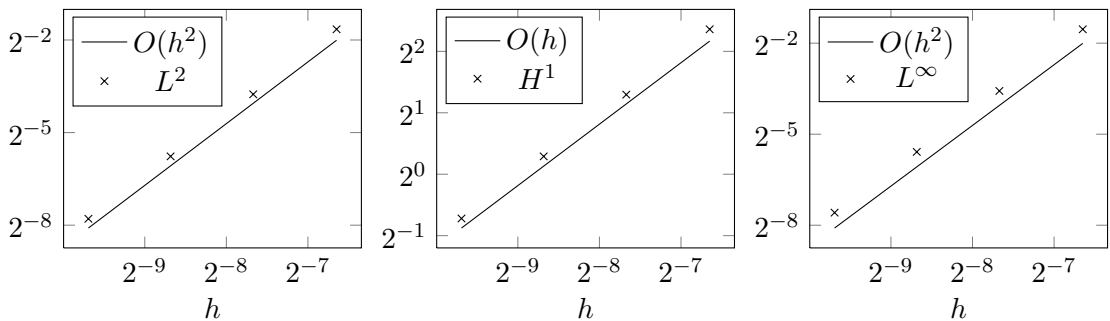
\begin{figure}[htb]
        \centering
        \begin{tikzpicture}%[width=.3\textwidth]
                \begin{loglogaxis}[
                        xlabel=$h$,
%                         ylabel=$\log_{10}{\Phi}$,
                        width=5.5cm,
                        log basis x=2,
                        log basis y=2,
                        legend pos=north west
                        ]
                        \addplot+[color=black,mark=]              table[header=true,x expr=1/(sqrt(\thisrowno{0})),y expr=1/(\thisrowno{0})*2500]{\ratesinhloc};
                        \addplot+[only marks,mark=x,color=black]  table[header=true,x expr=1/(sqrt(\thisrowno{0})),y index=1]{\ratesinhloc};
                        \legend{$O(h^2)$,$L^2$}
                \end{loglogaxis}
        \end{tikzpicture}
        \begin{tikzpicture}%[width=.3\textwidth]
                \begin{loglogaxis}[
                        xlabel=$h$,
%                         ylabel=$\log_{10}{\Phi}$,
                        width=5.5cm,
                        log basis x=2,
                        log basis y=2,
                        legend pos=north west
                        ]
                        \addplot+[color=black,mark=]              table[header=true,x expr=1/(sqrt(\thisrowno{0})),y expr=(1/sqrt(\thisrowno{0}))*450]{\ratesinhloc};
                        \addplot+[only marks,mark=x,color=black]  table[header=true,x expr=1/(sqrt(\thisrowno{0})),y index=2]{\ratesinhloc};
                        \legend{$O(h)$,$H^1$}
                \end{loglogaxis}
        \end{tikzpicture}
        \begin{tikzpicture}%[width=.3\textwidth]
                \begin{loglogaxis}[
                        xlabel=$h$,
%                         ylabel=$\log_{10}{\Phi}$,
                        width=5.5cm,
                        log basis x=2,
                        log basis y=2,
                        legend pos=north west
                        ]
                        \addplot+[color=black,mark=]              table[header=true,x expr=1/(sqrt(\thisrowno{0})),y expr=1/(\thisrowno{0})*2500]{\ratesinhloc};
                        \addplot+[only marks,mark=x,color=black]  table[header=true,x expr=1/(sqrt(\thisrowno{0})),y index=3]{\ratesinhloc};
                        \legend{$O(h^2)$,$L^\infty$}
                \end{loglogaxis}
        \end{tikzpicture}
        \caption{\label{fig:convergence_3} Convergence rates in $h$: Locally refined mesh with number of vertices in $\{10\,053, 41\,445, 168\,717, 681\,957\}$ and fixed $\eps=1/2^{20}$. The solid lines correspond to the predicted rates ($h^2$, $h$, $h^2$ from left to right). The actual errors $\|u-u^\eps_h\|_{L^2(\Omega)}$, $\|u-u^\eps_h\|_{H^1(\Omega)}$, $\|u-u^\eps_h\|_{L^\infty(\Omega)}$ (from left to right)  are marked by crosses.}
\end{figure}
% \begin{figure}
%  \centering
%  \includegraphics[width=.32\textwidth]{l2_h_loc}
%  \includegraphics[width=.32\textwidth]{h1_h_loc}
%  \includegraphics[width=.32\textwidth]{linf_h_loc}
%  \caption{\label{fig:convergence_3} Convergence rates in $h$: Locally refined mesh with number of vertices in $\{10\,053, 41\,445, 168\,717, 681\,957\}$ and fixed $\eps=1/2^{20}$. The solid lines correspond to the predicted rates ($h^2$, $h$, $h^2$ from left to right). The dotted lines are the actual errors $\|u-u^\eps_h\|_{L^2(\Omega)}$, $\|u-u^\eps_h\|_{H^1(\Omega)}$, $\|u-u^\eps_h\|_{L^\infty(\Omega)}$ (from left to right).}
% \end{figure}

All these convergence rates are predicted by Theorem~\ref{thm:main3}, cf. Theorem~\ref{thm:errorlinfty_h} for the $L^\infty$-estimates on the continuous level. 
For locally refined meshes around $\partial\Se$ such that $\delta=h^2$ we thus recover the convergence rates of the usual finite element method when used in combination with aligned grids.
The example presented here also shows that these convergence rates are sharp in general.
The number of degrees of freedom roughly doubles in two spatial dimensions compared to the corresponding uniform meshes, i.e. the computational complexity increases only by a multiplicative factor independent of the degrees of freedom. For $n=3$, the situation is worse regarding the number of degrees of freedom and the results presented here might be extended using anisotropic finite elements, see e.g. \cite{Apel99}.
%

% \subsection{Second Example}\label{sec:example2}

\section{Conclusions}\label{sec:conclusion}
In this paper a diffuse interface method for solving Poisson's equation with embedded interface conditions has been investigated.
The diffuse interface method could be interpreted as a standard Dirichlet problem on approximating domains. 
Thus, this method can also be used to numerically solve Poisson's equation with Dirichlet boundary conditions on complicated domains.
Error estimates in $H^1$-, $L^2$- and $L^\infty$-norms have been proven and verified numerically.
The use of uniform meshes gave suboptimal convergence rates in terms of the mesh size.
This is explained by the fact that the approximating domains are polygonal and non-convex in general, and therefore the corresponding Dirichlet problems do not allow for $H^2$-regular solutions in general.
Using locally refined meshes we could recover order-optimal convergence rates in terms of the mesh size. 
In two spatial dimensions, the computational complexity is only increased by a constant factor when using locally refined meshes. For three dimensional problems, the analysis might be extended to cover anisotropic finite elements, which leads to a computationally efficient method.
The presented method might furthermore be applied to more general second order elliptic, and the generalization to parabolic equations with $\Gamma=\Gamma(t)$ seems possible.
Moreover, the method might be extended to interface problems where $g=g(u)$ couples the two subproblems investigated here. 
It is open, whether one can improve the method, for instance by a post-processing step, in order to retrieve order-optimal convergence with respect to the mesh size (at least away from the interface) in the case when quasi-uniform grids are used.

\section*{Acknowledgements}
The author thanks Prof. Herbert Egger and the anonymous referees for constructive suggestions when finishing this work.
Support by ERC via Grant EU FP 7 - ERC Consolidator Grant 615216 LifeInverse is gratefully acknowledged. 

% \bibliographystyle{abbrv}
% \bibliography{interface_bib}

\begin{thebibliography}{10}

\bibitem{AbelsLamStinner2015}
H.~Abels, K.~F. Lam, and B.~Stinner.
\newblock Analysis of the diffuse domain approach for a bulk-surface coupled
  pde system.
\newblock {\em SIAM Journal on Mathematical Analysis}, 47(5):3687--3725, 2015.

\bibitem{Adams1975}
R.~A. Adams.
\newblock {\em Sobolev spaces}.
\newblock Academic Press [A subsidiary of Harcourt Brace Jovanovich,
  Publishers], New York-London, 1975.
\newblock Pure and Applied Mathematics, Vol. 65.

\bibitem{Apel99}
T.~Apel.
\newblock {\em Anisotropic finite elements: local estimates and applications}.
\newblock Advances in Numerical Mathematics. B. G. Teubner, Stuttgart, 1999.

\bibitem{Arrieta08}
J.~M. Arrieta, A.~Rodr{\'{\i}}guez-Bernal, and J.~D. Rossi.
\newblock The best {S}obolev trace constant as limit of the usual {S}obolev
  constant for small strips near the boundary.
\newblock {\em Proc. Roy. Soc. Edinburgh Sect. A}, 138(2):223--237, 2008.

\bibitem{Babuska72}
I.~Babu{\v{s}}ka.
\newblock The finite element method with {L}agrangian multipliers.
\newblock {\em Numer. Math.}, 20:179--192, 1972/73.

\bibitem{BankYserentant2014}
R.~E. Bank and H.~Yserentant.
\newblock On the {$H^1$}-stability of the {$L_2$}-projection onto finite
  element spaces.
\newblock {\em Numer. Math.}, 126(2):361--381, 2014.

\bibitem{BarrettElliot1987}
J.~W. Barrett and C.~M. Elliott.
\newblock Fitted and unfitted finite-element methods for elliptic equations
  with smooth interfaces.
\newblock {\em IMA J. Numer. Anal.}, 7(3):283--300, 1987.

\bibitem{BaCaDo2004}
S.~Bartels, C.~Carstensen, and G.~Dolzmann.
\newblock Inhomogeneous {D}irichlet conditions in a priori and a posteriori
  finite element error analysis.
\newblock {\em Numer. Math.}, 99(1):1--24, 2004.

\bibitem{BastianEngwer2009}
P.~Bastian and C.~Engwer.
\newblock An unfitted finite element method using discontinuous {G}alerkin.
\newblock {\em Internat. J. Numer. Methods Engrg.}, 79(12):1557--1576, 2009.

\bibitem{Bertalmio2001}
M.~Bertalm{\'{\i}}o, L.-T. Cheng, S.~Osher, and G.~Sapiro.
\newblock Variational problems and partial differential equations on implicit
  surfaces.
\newblock {\em J. Comput. Phys.}, 174(2):759--780, 2001.

\bibitem{Blair1973}
J.~J. Blair.
\newblock Bounds for the change in the solutions of second order elliptic
  {PDE}'s when the boundary is perturbed.
\newblock {\em SIAM J. Appl. Math.}, 24:277--285, 1973.

\bibitem{BrambleDupontThomee1972}
J.~H. Bramble, T.~Dupont, and V.~Thom{\'e}e.
\newblock Projection methods for {D}irichlet's problem in approximating
  polygonal domains with boundary-value corrections.
\newblock {\em Math. Comp.}, 26:869--879, 1972.

\bibitem{BrambleKing1994}
J.~H. Bramble and J.~T. King.
\newblock A robust finite element method for nonhomogeneous {D}irichlet
  problems in domains with curved boundaries.
\newblock {\em Math. Comp.}, 63(207):1--17, 1994.

\bibitem{BrambleKing1997}
J.~H. Bramble and J.~T. King.
\newblock A finite element method for interface problems in domains with smooth
  boundaries and interfaces.
\newblock {\em Adv. Comput. Math.}, 6(2):109--138 (1997), 1996.

\bibitem{BrennerScott}
S.~C. Brenner and L.~R. Scott.
\newblock {\em The mathematical theory of finite element methods}, volume~15 of
  {\em Texts in Applied Mathematics}.
\newblock Springer, New York, third edition, 2008.

\bibitem{BrezziFortin}
F.~Brezzi and M.~Fortin.
\newblock {\em Mixed and hybrid finite element methods}, volume~15 of {\em
  Springer Series in Computational Mathematics}.
\newblock Springer-Verlag, New York, 1991.

\bibitem{BES2014}
M.~Burger, O.~Elvetun, and M.~Schlottbom.
\newblock Analysis of the diffuse domain method for second order elliptic
  boundary value problems.
\newblock {\em Foundations of Computational Mathematics}, pages 1--48, 2015.

\bibitem{BES2015}
M.~Burger, O.~L. Elvetun, and M.~Schlottbom.
\newblock Diffuse interface methods for inverse problems: case study for an
  elliptic cauchy problem.
\newblock {\em Inverse Problems}, 31(12):125002, 2015.

\bibitem{ChenZou1998}
Z.~Chen and J.~Zou.
\newblock Finite element methods and their convergence for elliptic and
  parabolic interface problems.
\newblock {\em Numer. Math.}, 79(2):175--202, 1998.

\bibitem{Chernyshenko2013}
A.~Y. Chernyshenko and M.~A. Olshanskii.
\newblock Non-degenerate {E}ulerian finite element method for solving {PDE}s on
  surfaces.
\newblock {\em Russian J. Numer. Anal. Math. Modelling}, 28(2):101--124, 2013.

\bibitem{Ciarlet2002}
P.~G. Ciarlet.
\newblock {\em The finite element method for elliptic problems}, volume~40 of
  {\em Classics in Applied Mathematics}.
\newblock Society for Industrial and Applied Mathematics (SIAM), Philadelphia,
  PA, 2002.
\newblock Reprint of the 1978 original [North-Holland, Amsterdam; MR0520174 (58
  \#25001)].

\bibitem{DelfourZolesio2011}
M.~C. Delfour and J.-P. Zol{\'e}sio.
\newblock {\em Shapes and geometries}, volume~22 of {\em Advances in Design and
  Control}.
\newblock Society for Industrial and Applied Mathematics (SIAM), Philadelphia,
  PA, second edition, 2011.
\newblock Metrics, analysis, differential calculus, and optimization.

\bibitem{ElliottStinner2009}
C.~M. Elliott and B.~Stinner.
\newblock Analysis of a diffuse interface approach to an advection diffusion
  equation on a moving surface.
\newblock {\em Math. Models Methods Appl. Sci.}, 19(5):787--802, 2009.

\bibitem{ElliotStinnerStylesWelford2011}
C.~M. Elliott, B.~Stinner, V.~Styles, and R.~Welford.
\newblock Numerical computation of advection and diffusion on evolving diffuse
  interfaces.
\newblock {\em IMA J. Numer. Anal.}, 31(3):786--812, 2011.

\bibitem{FranzGaertnerRoosVoigt2012}
S.~Franz, R.~G{\"a}rtner, H.-G. Roos, and A.~Voigt.
\newblock A note on the convergence analysis of a diffuse-domain approach.
\newblock {\em Comput. Methods Appl. Math.}, 12(2):153--167, 2012.

\bibitem{GT2001}
D.~Gilbarg and N.~S. Trudinger.
\newblock {\em Elliptic partial differential equations of second order}.
\newblock Classics in Mathematics. Springer-Verlag, Berlin, 2001.
\newblock Reprint of the 1998 edition.

\bibitem{GiraultGlowinski1995}
V.~Girault and R.~Glowinski.
\newblock Error analysis of a fictitious domain method applied to a {D}irichlet
  problem.
\newblock {\em Japan J. Indust. Appl. Math.}, 12(3):487--514, 1995.

\bibitem{GlowinskiPanPeriaux1994}
R.~Glowinski, T.-W. Pan, and J.~P{\'e}riaux.
\newblock A fictitious domain method for {D}irichlet problem and applications.
\newblock {\em Comput. Methods Appl. Mech. Engrg.}, 111(3-4):283--303, 1994.

\bibitem{Grisvard85}
P.~Grisvard.
\newblock {\em Elliptic Problems in Nonsmooth Domains}.
\newblock Pitman, Boston, 1985.

\bibitem{HackbuschSauter1997}
W.~Hackbusch and S.~A. Sauter.
\newblock Composite finite elements for the approximation of {PDE}s on domains
  with complicated micro-structures.
\newblock {\em Numer. Math.}, 75(4):447--472, 1997.

\bibitem{HansboHansbo2002}
A.~Hansbo and P.~Hansbo.
\newblock An unfitted finite element method, based on {N}itsche's method, for
  elliptic interface problems.
\newblock {\em Comput. Methods Appl. Mech. Engrg.}, 191(47-48):5537--5552,
  2002.

\bibitem{LervagLowengrub2014}
K.~Y. Lervag and J.~Lowengrub.
\newblock Analysis of the diffuse-domain method for solving {PDE}s in complex
  geometries.
\newblock {\em Communications in Mathematical Sciences}, 13(6):1473--1500,
  2015.

\bibitem{LeVequeLin1994}
R.~J. LeVeque and Z.~L. Li.
\newblock The immersed interface method for elliptic equations with
  discontinuous coefficients and singular sources.
\newblock {\em SIAM J. Numer. Anal.}, 31(4):1019--1044, 1994.

\bibitem{LiMelenkWohlmuthZou2010}
J.~Li, J.~M. Melenk, B.~Wohlmuth, and J.~Zou.
\newblock Optimal a priori estimates for higher order finite elements for
  elliptic interface problems.
\newblock {\em Appl. Numer. Math.}, 60(1-2):19--37, 2010.

\bibitem{LiLowengrubRaetzVoigt2009}
X.~Li, J.~Lowengrub, A.~R{\"a}tz, and A.~Voigt.
\newblock Solving {PDE}s in complex geometries: a diffuse domain approach.
\newblock {\em Commun. Math. Sci.}, 7(1):81--107, 2009.

\bibitem{Li2003}
Z.~Li.
\newblock An overview of the immersed interface method and its applications.
\newblock {\em Taiwanese J. Math.}, 7(1):1--49, 2003.

\bibitem{LiLinWu2003}
Z.~Li, T.~Lin, and X.~Wu.
\newblock New {C}artesian grid methods for interface problems using the finite
  element formulation.
\newblock {\em Numer. Math.}, 96(1):61--98, 2003.

\bibitem{LiehrPreusserRumpfSauterSchwen2009}
F.~Liehr, T.~Preusser, M.~Rumpf, S.~Sauter, and L.~O. Schwen.
\newblock Composite finite elements for 3{D} image based computing.
\newblock {\em Comput. Vis. Sci.}, 12(4):171--188, 2009.

\bibitem{LinLinZhang2015}
T.~Lin, Y.~Lin, and X.~Zhang.
\newblock Partially {P}enalized {I}mmersed {F}inite {E}lement {M}ethods {F}or
  {E}lliptic {I}nterface {P}roblems.
\newblock {\em SIAM J. Numer. Anal.}, 53(2):1121--1144, 2015.

\bibitem{MaitreTomas1999}
J.-F. Maitre and L.~Tomas.
\newblock A fictitious domain method for {D}irichlet problems using mixed
  finite elements.
\newblock {\em Appl. Math. Lett.}, 12(4):117--120, 1999.

\bibitem{ParvizianDuesterRank2007}
J.~Parvizian, A.~D{\"u}ster, and E.~Rank.
\newblock Finite cell method: {$h$}- and {$p$}-extension for embedded domain
  problems in solid mechanics.
\newblock {\em Comput. Mech.}, 41(1):121--133, 2007.

\bibitem{Peskin1977}
C.~S. Peskin.
\newblock Numerical analysis of blood flow in the heart.
\newblock {\em J. Computational Phys.}, 25(3):220--252, 1977.

\bibitem{ReuterHillHarrison2012}
M.~G. Reuter, J.~C. Hill, and R.~J. Harrison.
\newblock Solving {PDE}s in irregular geometries with multiresolution methods
  {I}: {E}mbedded {D}irichlet boundary conditions.
\newblock {\em Comput. Phys. Commun.}, 183(1):1--7, 2012.

\bibitem{Thomee1973}
V.~Thom{\'e}e.
\newblock Polygonal domain approximation in {D}irichlet's problem.
\newblock {\em J. Inst. Math. Appl.}, 11:33--44, 1973.

\bibitem{WuLiLai2011}
C.-T. Wu, Z.~Li, and M.-C. Lai.
\newblock Adaptive mesh refinement for elliptic interface problems using the
  non-conforming immersed finite element method.
\newblock {\em Int. J. Numer. Anal. Model.}, 8(3):466--483, 2011.

\end{thebibliography}

\begin{appendix}
\section{Basic Properties of Diffuse Approximations}
We let $D=D_1$ and $\Gamma=\partial D$ in the following. The corresponding estimates for $D_2$ are derived in a similar fashion.
Let $\eps_0>0$, $0<\eps\leq \eps_0$ and 
$$
 \De = \{x\in D: {\rm dist}(x,\Gamma)> \eps\},\quad \Gamma_\eps=D\setminus D^{\eps}.
$$
For $\eps_0$ sufficiently small, the projection $p_{\Gamma}: \Gamma_{\eps_0}\to \Gamma$ is well-defined, and given by \cite[Chapter 7, Theorem 3.1]{DelfourZolesio2011}
\begin{align}\label{eq:projection}
 p_{\Gamma}(x) = x- d_\Gamma(x) \nabla d_\Gamma(x).%,\quad x\in\mathcal{S}^{\eps_0}.
\end{align}
Notice that $\nabla d_\Gamma(x)= n(p_{\partial D}(x))$ for all $x\in \Gamma_{\eps_0}$ \cite[Chapter 7, Theorem~8.5]{DelfourZolesio2011}, and we will therefore write $n(x)=\nabla d_\Gamma(x)$.
Here, $n$ denotes the exterior unit normal field to $\partial D$.
For $t\in [0,\eps_0]$, we define the mapping $\Phi_t(x)=x-tn(x)$, $x\in \Gamma$, and note that $\Phi_t(\Gamma)=\partial D^t$. Moreover, $D\Phi_t(x)=I-tD^2 d_\Gamma(x)$, and, cf. \cite[Eq. (9)]{BES2014},
\begin{align}\label{eq:det_to_one}
 \sup_{x\in\Gamma} |\det D \Phi_t(x)-1+t \Delta d_D(x)| \leq C t^2.
\end{align}
% In view of \eqref{eq:det_to_one}, we may also assume that
Hence, after decreasing $\eps_0$ if necessary, we may assume that
\begin{align}\label{eq:det_half}
 \frac{1}{2} \leq \det D\Phi_t \leq 2,\quad 0\leq t\leq \eps_0.
\end{align}
For any integrable $v$ the transformation formula implies
\begin{align}\label{eq:fubini}
   \int_{\Gamma_{\eps}} v(x)\d x = \int_{\Gamma}\int_{0}^\eps v(x-tn(x)) |\det D\Phi_t(x)| \d t\d\sigma(x).
\end{align}
For the right-hand side we will employ the fundamental theorem of calculus
\begin{align}\label{eq:fundamental}
 v(x- t n(x)) = v(x) - \int_0^t \partial_n v(x-s n(x)) \d s,\quad 0\leq t\leq \eps_0.
\end{align}
The following is a central estimate; cf. \cite[Theorem~A.2]{BES2015}. 
\begin{theorem}\label{thm:estimate} 
 There exists a constant $C>0$ not depending on $\eps$ such that
 \begin{align*}
  \|v\|_{L^2(\Gamma_{\eps})} \leq C(\sqrt{\eps}\|v\|_{L^2(\Gamma)} + \eps\|\partial_n v\|_{L^2(\Gamma_{\eps})}) \quad\text{ for $v\in H^1(\Gamma_{\eps})$.}
 \end{align*}
\end{theorem}
\begin{proof}
  Let $0\leq t\leq \eps$. Using \eqref{eq:fundamental} and H\"older's inequality we obtain
  \begin{align*}
  |v(x- tn(x))|^2 \leq 2 ( |v(x)|^2 + \eps\int_{0}^{\eps} |\partial_n v(x-sn(x))|^2 \d s).
  \end{align*}
 Using the latter in \eqref{eq:fubini} and using \eqref{eq:det_half}, we obtain
\begin{align*}
  \|v\|_{L^2(\Gamma_{\eps})}^2 \leq C(\eps\|v\|_{L^2(\Gamma)}^2+ \eps^2 \int_{\Gamma} \int_{0}^{\eps} |\partial_n v(x-sn(x))|^2\d s\d\sigma).
\end{align*}
Using \eqref{eq:fubini} and \eqref{eq:det_half} we further may write
\begin{align*}
  \int_{\Gamma}\int_{0}^{\eps} |\partial_n v(x-sn(x))|^2\d s\d\sigma\leq C\int_{\Gamma_{\eps}} |\partial_n v(x)|^2\d x,
\end{align*}
which completes the proof.
\end{proof}

\begin{lemma}\label{lem:error_perturbed_interface}
 There exists a constant $C>0$ independent of $\eps$ such that for $v\in H^1(\Gamma_{\eps})$
 \begin{align*}
 \| v\|_{L^2(\partial D^\eps)} \leq C (\|v\|_{L^2(\Gamma)} + \sqrt{\eps}\|\partial_n v\|_{L^2(\Gamma_{\eps})} ).
\end{align*}
\end{lemma}
\begin{proof}
 The transformation formula yields
 \begin{align*}
  \int_{\partial D^\eps} |v|^2\d\sigma &= \int_{\Gamma} |v(x-\eps n(x))|^2 |\det D\Phi_{\eps}(x)|\d\sigma(x).
 \end{align*}
 Using \eqref{eq:fundamental}, the assertion follows with similar arguments as in the proof of Theorem~\ref{thm:estimate}.
\end{proof}

% Lemma~\ref{lem:error_perturbed_interface} yields a uniform trace estimate.
\begin{lemma}\label{lem:trace_unweighted}
 There is a constant $C>0$ independent of $\eps$ such that for $v\in H^1(D^\eps)$
 \begin{align*}%\label{eq:uniform_trace}
  \|v\|_{L^2(\partial D^{\eps})} \leq C \|v\|_{H^1(D^\eps)}.%\int_{D_i} |\nabla v|^p + |v|^p \d x.
\end{align*}
\end{lemma}
\begin{proof}
We extend $v\in H^1(D^\eps)$ to a function $v\in H^1(\RR^n)$ by reflection \cite[Theorem~4.26]{Adams1975}. We have $\|v\|_{H^1(\RR^n)}\leq C \|v\|_{H^1(D^\eps)}$.
Continuity of the embedding $H^1(D)\hookrightarrow L^2(\partial D)$ \cite{Adams1975} implies $\|v\|_{L^2(\partial D)}\leq C \|v\|_{H^1(D)}$. Lemma~\ref{lem:error_perturbed_interface} and $\|v\|_{H^1(\Gamma_\eps)}\leq \|v\|_{H^1(D)}$ complete the proof.
\end{proof}
The preceding lemma with a different proof may also be found in \cite[Lemma~2.1]{Arrieta08}.
\end{appendix}
\end{document}